\documentclass[a4paper, 10pt, american]{amsart}

\usepackage{times,latexsym,amssymb}
\usepackage{amsmath,amsthm,bm}
\usepackage{color}
\usepackage[colorlinks,pdfpagelabels,pdfstartview = FitH,bookmarksopen
= true,bookmarksnumbered = true,linkcolor = blue,plainpages =
false,hypertexnames = false,citecolor = red,pagebackref=false]{hyperref}

\usepackage{amsbsy}
\usepackage{amstext}
\usepackage{amssymb}
\usepackage{esint}
\usepackage{stmaryrd}
\usepackage[pdftex]{graphicx}
\usepackage{floatflt}

\allowdisplaybreaks
\newtheorem{theorem}{Theorem}
\newtheorem{lemma}[theorem]{Lemma}
\newtheorem{definition}[theorem]{Definition}
\newtheorem{proposition}[theorem]{Proposition}
\newtheorem{corollary}[theorem]{Corollary}
\newtheorem{remark}[theorem]{Remark}
\numberwithin{theorem}{section}
\numberwithin{equation}{section}

\newcommand{\mint}{- \mskip-19,5mu \int}

\def\N{\mathbb{N}}
\def\R{\mathbb{R}}

\renewcommand{\d}{\mathrm{d}}
\newcommand{\dx}{\mathrm{d}x}

\newcommand{\dt}{\mathrm{d}t}

\renewcommand{\epsilon}{\varepsilon}

\DeclareMathOperator{\Div}{div}

\renewcommand{\epsilon}{\varepsilon}
\newcommand{\eps}{\varepsilon}
\renewcommand{\rho}{\varrho}

\def\eqn#1$$#2$${\begin{equation}\label#1#2\end{equation}}

\newcommand{\power}[2]{\bm{#1^{\mbox{\unboldmath{\scriptsize$#2$}}}}}

\newcommand{\abs}[1]{|#1|}
\newcommand{\babs}[1]{\big|#1\big|}

\def\Xint#1{\mathchoice
    {\XXint\displaystyle\textstyle{#1}}%
    {\XXint\textstyle\scriptstyle{#1}}%
    {\XXint\scriptstyle\scriptscriptstyle{#1}}%
    {\XXint\scriptscriptstyle\scriptscriptstyle{#1}}%
    \!\int}
\def\XXint#1#2#3{\setbox0=\hbox{$#1{#2#3}{\int}$}
    \vcenter{\hbox{$#2#3$}}\kern-0.5\wd0}
\def\bint{\Xint-}
\def\dashint{\Xint{\raise4pt\hbox to7pt{\hrulefill}}}

\def\Xiint#1{\mathchoice
    {\XXiint\displaystyle\textstyle{#1}}%
    {\XXiint\textstyle\scriptstyle{#1}}%
    {\XXiint\scriptstyle\scriptscriptstyle{#1}}%
    {\XXiint\scriptscriptstyle\scriptscriptstyle{#1}}%
    \!\iint}
\def\XXiint#1#2#3{\setbox0=\hbox{$#1{#2#3}{\iint}$}
    \vcenter{\hbox{$#2#3$}}\kern-0.5\wd0}
\def\biint{\Xiint{-\!-}}

\subjclass[2010]{35B65, 35K67, 35K40, 35K55}

\begin{document}
\title[Higher integrability for the singular porous medium system]{Higher integrability for the\\ singular porous medium system}
\date{\today}

\author[V. B\"ogelein]{Verena B\"{o}gelein}
\address{Verena B\"ogelein\\
Fachbereich Mathematik, Universit\"at Salzburg\\
Hellbrunner Str. 34, 5020 Salzburg, Austria}
\email{verena.boegelein@sbg.ac.at}

\author[F. Duzaar]{Frank Duzaar}
\address{Frank Duzaar\\
Department Mathematik, Universit\"at Erlangen--N\"urnberg\\
Cauerstrasse 11, 91058 Erlangen, Germany}
\email{frank.duzaar@fau.de}

\author[C. Scheven]{Christoph Scheven}
\address{Christoph Scheven\\ Fakult\"at f\"ur Mathematik, 
Universit\"at Duisburg-Essen\\45117 Essen, Germany}
\email{christoph.scheven@uni-due.de}

\maketitle


\begin{abstract}
In this paper we establish in the fast diffusion range the higher integrability of the spatial gradient of weak solutions to porous medium systems. The result comes along with an explicit reverse H\"older inequality for the gradient.
The novel feature in the proof is a suitable intrinsic scaling for
space-time cylinders combined with reverse H\"older inequalities and a Vitali covering argument within this geometry. The main result
holds for the natural range of parameters suggested by other regularity results.  Our result applies to general fast diffusion systems and includes both,
nonnegative and signed solutions in the case
of equations. The methods of proof are purely vectorial in their structure.
\end{abstract}

\section{Introduction and results}
In this paper we study regularity of solutions to second-order
parabolic systems
\begin{equation}\label{general-PME}
  \partial_t u -\Div \mathbf A\big(x,t,u,D\big(|u|^{m-1}u\big)\big) =\Div F
\end{equation}
on a space-time cylinder  $\Omega_T:= \Omega\times (0,T)$ over a bounded domain
$\Omega\subset\R^n$, $n\in\N$, and $T>0$. Precise structural assumptions for the vector field
$\mathbf A$ are presented later. The principal prototype  is the inhomogeneous porous medium system
\begin{equation}\label{PME} 
  \partial_t u - \Delta\big(|u|^{m-1}u\big) = \Div F,
\end{equation}
with $m>0$. As usual, solutions to \eqref{general-PME} are taken in a weak sense, i.e.~they are assumed to belong to a parabolic Sobolev space whose
amount of integrability is determined by the growth of the vector
field $\mathbf{A}$ with respect to the gradient variable, cf.~Definition \ref{def:weak_solution}.
With the choice $m=1$  we recover the heat
equation. Equation \eqref{PME} has a different behavior when $m>1$ or $m<1$. The first case is called
slow diffusion range, since disturbances propagate with finite speed and free boundaries may occur, while
in the second case disturbances propagate with infinite speed and 
extinction in finite time may occur. This range is called fast diffusion range. 
For more information on the theory for the porous medium equation and related regularity results we refer to \cite{CaVaWo, DiBenedetto_Holder, Vazquez-1, Vazquez-2} and the references therein.

The main purpose of this paper is to establish a higher
integrability result for the gradient of weak solutions of porous
medium equations and systems of the type \eqref{general-PME} in the fast diffusion range.
More precisely, we show that there exists a universal constant $\epsilon >0$, such that
\begin{equation}\label{hi-int-PME}
  	D\big(|u|^{m-1}u\big)
	\in 
	L^{2+\eps}_{\rm loc} ,
\end{equation}
whenever $u$ is a weak solution to  \eqref{general-PME}, thereby ensuring that for weak solutions $u$
of the porous medium system the spatial gradient of $|u|^{m-1}u$  belongs to a slightly better Lebesgue space than the natural energy space $L^2$. This implies that porous medium systems as in \eqref{general-PME} possess the self-improving property of integrability. Our result comes along with a quantitative local
reverse H\"older type estimate for $|D(|u|^{m-1}u)|$; see Theorem \ref{thm:higherint}. 
The higher integrability for   porous medium systems as in  \eqref{general-PME} has been an open problem for a long time, even in the case of equations and non-negative solutions. Here we give a positive answer
in the  fast diffusion range 
\begin{equation}\label{lower-bound}
  m_c:=\frac{(n-2)_+}{n+2}<m\le 1.
\end{equation}
The lower bound on $m$ is natural and appears also in
other regularity results for porous medium equations, cf.~the
discussion in \cite[\S 6.21]{DBGV-book}.  For example,  solutions might be unbounded in the super-critical range $0<m\le m_c$.

The central idea in the proof of our main result is a new kind of intrinsic geometry. Until now, variants of 
this idea have been  successfully used in establishing  the self-improving property of integrability for the parabolic $p$-Laplacian  system \cite{Kinnunen-Lewis:1} and very recently  in the slow diffusion range $m\ge 1$ for the porous medium equation \cite{Gianazza-Schwarzacher} and  system \cite{BDKS-higher-int}. 
The central idea here is the construction of  suitable intrinsic cylinders $	Q_{r,s}(z_o):=B_r(x_o)\times(t_o-s,t_o+s)$ with $z_o=(x_o,t_o)$. Since the equation is nonlinear with respect to
$u$, we use cylinders whose space-time scaling depends on the mean values of  $|u|^{1+m}$. This choice is dictated by the leading term on the right-hand side in the energy estimate, which is of order $1+m$; cf.~Lemma~\ref{lem:energy}. This heuristic argument motivates to consider space-time cylinders $Q_{r,s}(z_o)$, such that the quotient $\frac{s}{r^{\frac{1+m}{m}}}$ satisfies
\begin{equation}\label{scaling-intro}
  \frac{s}{r^{\frac{1+m}{m}}}=\theta^{1-m},
  \mbox{ with}\quad
  \theta^{1+m}
  \approx
  \biint_{Q_{r,s}(z_o)} 
  \frac{|u|^{1+m}}{r^{\frac{1+m}m}}\,\dx\dt .
\end{equation}
In this  geometry the only ingredients for the proof of parabolic Sobolev-Poincar\'e and reverse H\"older type inequalities are the standard energy estimate and a gluing lemma; see Section~\ref{sec:energy}.  The construction of a system of such intrinsic cylinders is quite involved, since the cylinders on the right-hand side
\eqref{scaling-intro} also depend on the parameter $\theta$. 
In fact, we have to distinguish between two regimes, the non-singular and the singular regime. The first is characterized by the fact that cylinders are intrinsic, the latter 
by the fact that cylinders are only sub-intrinsic, which means that \eqref{scaling-intro}$_2$ only holds as an inequality where  the mean value integral  is bounded from above by $\theta^{1+m}$.
In both regimes we need to establish reverse H\"older type inequalities. In the actual construction of the cylinders, we modify the argument from \cite{Gianazza-Schwarzacher}; see also \cite{BDKS-higher-int} which is better suited to our purposes here.

At this stage some words to classify our result in the history of the problem  of  higher integrability
are appropriate. In the stationary case of elliptic systems the self-improving property was
first observed by  Elcrat \& Meyers \cite{Meyers-Elcrat}, see also the monographs 
\cite[Chapter~V, Theorem 2.1]{Giaquinta:book} and \cite[Section
6.4]{Giusti:book} and the references therein. The first higher integrability result for parabolic systems
goes back to Giaquinta \& Struwe \cite[Theorem 2.1]{Giaquinta-Struwe}.  For parabolic systems with $p$-growth, whose 
principle prototype is the parabolic $p$-Laplacian system, the higher integrability of the gradient of weak solutions was established by Kinnunen \& Lewis
\cite{Kinnunen-Lewis:1} in the range $p>\frac{2n}{n+2}$. This lower bound is natural and
appears  also in other contexts in the regularity theory of parabolic $p$-Laplace type systems; cf.~the monograph \cite{DiBe}.  
In the meantime the result has been generalized  in various directions, such as global results and higher order parabolic systems with $p$-growth; see
\cite{Boegelein:1,Boegelein-Parviainen,Parviainen}. The corresponding problem for the porous medium equation turned out to
be more involved and remained open for a long time, even in the scalar case
for  non-negative solutions. Additionally to
the obvious anisotropic behavior  of the equation with respect to scalar multiplication of solutions, it is also not possible to
add constants to a solution without destroying the property of
being a solution. 
This  difficulty has recently been overcome by Gianazza \& Schwarzacher
\cite{Gianazza-Schwarzacher} who proved in the slow diffusion range $m\ge1$
that non-negative weak solutions of \eqref{general-PME} admit the
self-improving property of higher integrability of the gradient.
The main novelty in their proof is the use of a new intrinsic scaling.
Instead of scaling  cylinders with
respect to $|Du|$  as in the case of the parabolic $p$-Laplacian (cf.~\cite{DiBe} and the references therein),
they  work with cylinders which are intrinsically
scaled with respect to $u$. The proof, however, uses the method of expansion of positivity and therefore can not be extended to signed solutions, porous medium type systems and  the fast diffusion range.
A simpler and more flexible proof, which does not rely on the expansion of positivity and which covers both
signed solutions and porous medium systems is given in
\cite{BDKS-higher-int}. Finally, in \cite{BDKS-doubly} the higher integrability is shown for doubly nonlinear
parabolic systems, whose
prototype is 
\begin{equation*}
  \partial_t\big(|u|^{p-2}u\big)-\Div(|Du|^{p-2}Du)=\Div \big( |F|^{p-2}F\big).
\end{equation*}
In this equation
aspects of both the porous
medium equation and the parabolic $p$-Laplace equation play a role. Therefore the intrinsic scaling  has to take
into account the degeneracy of the system both with respect to the
gradient variable and with respect to the solution
itself. In \cite{BDKS-doubly} the higher integrability is established for exponents $p$ in the somewhat unexpected range
$\max\{\frac{2n}{n+2},1\}<p<\frac{2n}{(n-2)_+}$. The lower bound also appears for the parabolic $p$-Laplace system \cite{Kinnunen-Lewis:1}, while the upper bound corresponds exactly to the lower
bound in~\eqref{lower-bound} for the porous medium equation in the fast diffusion range.

We point out that independently of us,
Gianazza \& Schwarzacher \cite{Gianazza-Schwarzacher:m<1} proved the higher integrability result
in the scalar case for  nonnegative solutions  in the fast diffusion
range~\eqref{lower-bound}. In contrast to \cite{Gianazza-Schwarzacher:m<1}, we prove the higher integrability regardless of whether the
solution is non-negative or signed in the scalar case, or vector-valued in the case of  systems. In another point, our results are also different.
Instead of an
inhomogeneity given by a bounded function  $f$, we consider
a right-hand side in divergence form $\Div F$ with $F\in L^\sigma$ for some $\sigma>2$. In \cite{Gianazza-Schwarzacher:m<1} the boundedness assumption on $f$ is imposed to ensure that weak solutions are bounded. Here, we are able to deal with unbounded solutions.

\section{Notation and main result}

\subsection{Notations}
To keep formulations as simple as possible, we define the {\em
  power of a vector} or of a possibly {\it negative number}  by
$$
	\power{u}{\alpha}
	:=
	|u|^{\alpha-1}u,
	\quad\mbox{for $u\in\R^N$ and $\alpha>0$,}
$$
which in the case $u=0$  and $\alpha\in (0,1)$ we interpret as $\power{u}{\alpha}=0$. 
Throughout the paper we write $z_o=(x_o,t_o)\in \R^n\times\R$ for points in space-time. We use space-time cylinders 
\begin{equation}\label{def-Q}
	Q_\rho^{(\theta)}(z_o)
	:=
	B_\rho^{(\theta)}(x_o)\times\Lambda_\rho(t_o),
\end{equation}
where 
\begin{equation*}
  	B_\rho^{(\theta)}(x_o)
  	:=
  	\Big\{x\in\R^n: |x-x_o|<\theta^{\frac{m(m-1)}{1+m}}\rho\Big\}
\end{equation*}
and
\begin{equation*}
  	\Lambda_\rho(t_o)
  	:=
  	\big(t_o-\rho^{\frac{1+m}{m}},t_o+\rho^{\frac{1+m}{m}}\big)
\end{equation*}
with some scaling  parameter $\theta >0$. 
In the case $\theta =1$, we simply
omit the parameter in the notation and write
$$
	Q_\rho(z_o)
	:=
	B_\rho(x_o)\times\big(t_o-\rho^{\frac{1+m}{m}},t_o+\rho^{\frac{1+m}{m}}\big)
$$
instead of $Q_\rho^{(1)}(z_o)$.
If the center $z_o$ is clear from the context we omit it in the notation.

For a map $u\in L^1(0,T;L^1(\Omega,\R^N))$ and given measurable sets $A\subset\Omega$  and $E\subset\Omega_T$ with positive Lebesgue measure the slicewise mean $\langle u\rangle_{A}\colon (0,T)\to \R^N$ of $u$ on  $A$ is defined
by
\begin{equation*}
	\langle u\rangle_{A}(t)
	:=
	\mint_{A} u(t)\,\dx,
	\quad\mbox{for a.e.~$t\in(0,T)$,}
\end{equation*}
whereas  the mean value $(u)_{E}\in \R^N$ of $u$ on  $E$ is defined by
\begin{equation*}
	(u)_{E}
	:=
	\biint_{E} u\,\dx\dt.
\end{equation*}
Note that if $u\in C^0((0,T);L^2(\Omega,\R^N))$ the slicewise means are defined for any $t\in (0,T)$. 
If  $A$ is a ball $B_\rho^{(\theta)}(x_o)$,  we write $\langle
u\rangle_{x_o;\rho}^{(\theta)}(t):=\langle
u\rangle_{B_\rho^{(\theta)}(x_o)}(t)$. Similarly,   if $E$ is a
cylinder of the form $Q_\rho^{(\theta)}(z_o)$, we use the shorthand
notation $(u)^{(\theta)}_{z_o;\rho}:=(u)_{Q_\rho^{(\theta)}(z_o)}$.

\subsection{General Setting and Results}
We consider porous medium type systems of the form 
\begin{equation}\label{por-med-eq}
  \partial_t u -\Div \mathbf A(x,t,u,D\power{u}{m}) =\Div F
  \quad\mbox{in $\Omega_T$,}
\end{equation}
where $\mathbf A\colon \Omega_T\times\R^N\times\R^{Nn}\to \R^{Nn}$ is a Carath\'eodory vector field satisfying the following ellipticity and growth conditions that are
modeled after the prototype system \eqref{PME}.
For structural constants $0<\nu\le L<\infty$, we assume that
\begin{equation}\label{growth}
\left\{
\begin{array}{c}
	\mathbf A(x,t,u,\xi)\cdot\xi\ge \nu|\xi|^2\, ,\\[6pt]
	| \mathbf A(x,t,u,\xi)|\le L|\xi|,
\end{array}
\right.
\end{equation}
for a.e.~$(x,t)\in \Omega_T$ and any $(u,\xi)\in \R^N\times\R^{Nn}$.
To formulate the main result, we introduce the notion of {\it weak  solution}. 

\begin{definition}\label{def:weak_solution}\upshape
Let $m>0$ and $\mathbf A\colon \Omega_T\times \R^N\times\R^{Nn}\to\R^{Nn}$ be a vector field satisfying \eqref{growth} and $F\in L^{2}(\Omega_T,\R^{Nn})$. 
A function 
\begin{equation}\label{spaces}
	u\in C^0 \big((0,T); L^{1+m}(\Omega,\R^N)\big)
	\quad\mbox{with}\quad 
	\power{u}{m}\in L^2\big(0,T;W^{1,2}(\Omega,\R^N)\big)
\end{equation} 
is a \textit{weak solution} to the porous medium type system \eqref{por-med-eq} if and only if the identity
\begin{align}\label{weak-solution}
	\iint_{\Omega_T}\big[u\cdot\partial_t\varphi - \mathbf A(x,t,u,D\power{u}{m})\cdot D\varphi\big]\dx\dt
    	=
    	\iint_{\Omega_T} F\cdot D\varphi \,\dx\dt
\end{align}
holds true, for any testing function $\varphi\in
C_0^\infty(\Omega_T,\R^N)$. 
\qed
\end{definition}
Our main result reads as follows:

\begin{theorem}\label{thm:higherint}
Assume that 
$$
	m_c:=\frac{(n-2)_+}{n+2}<m\le 1
$$ 
and $\sigma>2$. Then, there exists $\eps_o=\eps_o(n,m,\nu,L)\in (0,1]$
such that whenever $F\in L^\sigma(\Omega_T,\R^{Nn})$ and $u$
is a weak solution of Equation~\eqref{por-med-eq}  in the sense of Definition~\ref{def:weak_solution}
under the assumptions~\eqref{growth}, then with  $\epsilon_1:=\min\{\eps_o,\sigma-2\}$ we have
$$
  	D\power{u}{m} 
	\in 
	L^{2+\eps_1}_{\rm loc}\big(\Omega_T,\R^{Nn}\big).
$$
Moreover, for every $\eps\in(0,\epsilon_1]$ and every cylinder
$
  Q_{2R}(z_o)\subseteq\Omega_T
$,
we have the quantitative local higher integrability estimate
\begin{align}\label{eq:higher-int}
	&\biint_{Q_{R}}
	|D\power{u}{m}|^{2+\epsilon} \d x\d t\nonumber \\
	&\qquad \le
	c
	\Bigg[
	1+\biint_{Q_{2R}} \bigg[\frac{|u|^{1+m}}{R^{\frac{1+m}{m}}} + |F|^{2}\bigg]
	\dx\dt
	\Bigg]^{\frac{\epsilon d}{2}}
	\biint_{Q_{2R}}
	|D\power{u}{m}|^{2} \,\dx\dt \nonumber\\
        &\qquad\qquad+
	c\,\biint_{Q_{2R}}  |F|^{2+\epsilon}
	\,\dx\dt .
      \end{align}
with $c=c(n,m,$ $\nu,L)\ge 1$. Here,  
\begin{equation}\label{def:d}
  	d
	:=
  	\frac{2(1+m)}{2(1+m)-n(1-m)}
\end{equation}
denotes the scaling deficit.
\end{theorem}

\begin{remark}
\textup{  The scaling deficit $d$ that appears in the higher integrability
  estimate reflects the inhomogeneous scaling behavior of the porous
  medium system. In particular, we have  $d=1$ if
  $m=1$, which corresponds to the case of the classical heat
  equation. On the other hand,  in the case $n\ge2$ we have
  $d\to\infty$ in the limit $m\downarrow m_c$.
  The
  latter 
  underlines the significance of the lower bound $m>m_c$. If $m=1$, then  \eqref{eq:higher-int} is similar to the reverse
  H\"older inequalities in \cite{Giaquinta-Struwe} and \cite{Kinnunen-Lewis:1} with  $p=2$.\qed
}\end{remark}
\begin{remark}
\textup{ The above higher integrability result can easily be extended to the case of vector-fields $\mathbf A$ satisfying  the more general growth and coercivity conditions 
\begin{equation*}
\left\{
\begin{array}{c}
	\mathbf A(x,t,u,\xi)\cdot\xi\ge \nu|\xi|^2-h_1\, ,\\[6pt]
	| \mathbf A(x,t,u,\xi)|\le L|\xi| +h_2,
\end{array}
\right.
\end{equation*}
with non-negative measurable functions $h_1,h_2\colon \Omega_T\to [0,\infty]$, so that
$h_1+ h_2^2\in L^{\sigma/2}(\Omega_T)$ for the exponent $\sigma>2$
from Theorem~\ref{thm:higherint}.
\qed
}\end{remark}

The  quantitative local estimate \eqref{eq:higher-int} can easily be converted into an estimate on standard
parabolic cylinders $C_R(z_o):=B_{R}(x_o)\times(t_o-R^2,t_o+R^2)$. The precise statement is:

\begin{corollary}\label{cor:higher-int}
  Under the assumptions of Theorem \ref{thm:higherint},    on any cylinder $C_{2R}(z_o)\subseteq\Omega_T$  and for every $\eps\in(0,\eps_1]$ we have
     \begin{align*}
	R^{2+\eps} &
	\biint_{C_{R}(z_o)}
    |D\power{u}{m}|^{2+\epsilon} \d x\d t \\
	&\le
	c\,R^2\bigg[
	1+\biint_{C_{2R}(z_o)} \big[\abs{u}^{1+m} + R^2|F|^2\big] \d x\d t 
	\bigg]^{\frac{\epsilon d}{2}}
	\biint_{C_{2R}(z_o)}
	|D\power{u}{m}|^{2} \dx\dt \\
	&\quad +
	c\, R^{2+\epsilon}
	\biint_{C_{2R}(z_o)} |F|^{2+\epsilon} \dx\dt ,
   \end{align*}
for  a constant $c=c(n,m,\nu,L)$.
\end{corollary}

\section{Auxiliary Material}
In this section we provide the necessary tools which will be used later.
To ``re-absorb'' certain terms, we frequently  shall use the following iteration lemma, cf. \cite[Lemma 6.1]{Giusti:book}.

\begin{lemma}\label{lem:tech}
Let $0<\vartheta<1$, $A,C\ge 0$ and $\alpha > 0$. Then there exists a constant  $c = c(\alpha,\vartheta)$
such that for any non-negative bounded function $\phi\colon[r,\rho]\to [0,\infty)$ with $0<r<\rho$ satisfying
\begin{equation*}
	\phi(t)
	\le
	\vartheta\, \phi(s) + \frac{A}{(s-t)^\alpha} + C
	\qquad \text{for all $r\le t<s\le \varrho$,}
\end{equation*}
we have
\begin{equation*}
	\phi(r)
	\le
	c\,  \bigg[\frac{A}{(\varrho - r)^\alpha} + C\bigg].
\end{equation*}
\end{lemma}

The following lemma can be deduced as in \cite[Lemma~8.3]{Giusti:book}.

\begin{lemma}\label{lem:Acerbi-Fusco}
For any $\alpha>0$, there exists a constant $c=c(\alpha)$ such that,
for all $a,b\in\R^N$, $N\in\N$, the following inequality holds true:
\begin{align*}
	\tfrac1c\big|\power{b}{\alpha} - \power{a}{\alpha}\big|
	\le
	\big(|a| + |b|\big)^{\alpha-1}|b-a|
	\le
	c \big|\power{b}{\alpha} - \power{a}{\alpha}\big|.
\end{align*}
\end{lemma}

The next lemma is an immediate consequence of Lemma~\ref{lem:Acerbi-Fusco}.

\begin{lemma}\label{lem:a-b}
For any $\alpha\ge 1$, there exists a constant $c=c(\alpha)$ such that,
for all $a,b\in\R^N$, $N\in\N$, the following inequality holds true:
\begin{align*}
	|b-a|^\alpha
    \le
    c\big|\power{b}{\alpha} - \power{a}{\alpha}\big|.
\end{align*}
\end{lemma}

It is well known that mean values over subsets $A\subset B$ are quasi-minimizers of the mapping $
\R^N\ni a\mapsto \int_B |u-a|^p \dx$. The following lemma shows that this also applies to powers 
 $\power{u}{\alpha}$  of $u$, provided  $\alpha\ge\frac1p$.  For $p=2$ and $A=B$, the lemma has been proved in \cite[Lemma 6.2]{Diening-Kaplicky-Schwarzacher}.  The general version is established
in \cite[Lemma~3.5]{BDKS-doubly}. 

\begin{lemma}\label{lem:alphalemma}
For any $p\ge 1$ and $\alpha\ge\frac1p$, there exists a universal constant $c=c(\alpha,p)$ such that whenever $A\subset B\subset \R^k$, $k\in\N$, are two bounded domains with positive measure, then for any  $u \in L^{\alpha p}(B,\R^N)$ and any  $a\in\R^N$, we have 
$$
	\mint_B \big|\power{u}{\alpha}-\power{(u)_A}{\alpha}\big|^p \dx 
	\le 
	\frac{c\,|B|}{|A|} 
	\mint_B \big|\power{u}{\alpha}-\power{a}{\alpha}\big|^p \dx.
$$
\end{lemma}

\section{Energy bounds}\label{sec:energy}
In this section we state an energy inequality and a gluing lemma. Both follow with standard arguments
from the weak form \eqref{weak-solution} of the differential equation by testing with suitable testing functions.  Later on, they will be used in the proof of Sobolev-Poincar\'e and reverse H\"older type inequalities.
At this point it should be emphasized, that these two lemmas are the
only places in the proof of the higher integrability where the porous
medium system is utilized. The proof of the energy estimate is along
the lines of \cite[Lemma~3.1]{BDKS-higher-int}, taking into account
\cite[Lemma~2.3\,(i)]{BDKS-higher-int} or
\cite[Lemma~3.4]{BDKS-doubly} and the different definition of scaled
cylinders. The latter means  that the radii $\rho$ and $r$ in
\cite[Lemma~3.1]{BDKS-higher-int} have to be replaced by
$\theta^{\frac{m(m-1)}{1+m}}\rho\,$ and $\theta^{\frac{m(m-1)}{1+m}}r$.

\begin{lemma}\label{lem:energy}
Let $m>0$ and $u$ be a weak solution to \eqref{por-med-eq} in $\Omega_T$ in the sense of Definition~{\upshape\ref{def:weak_solution}}.
Then, on any cylinder $Q_{\rho}^{(\theta)}(z_o)\subseteq\Omega_T$ with
$\rho, \theta>0$, for any $r\in[\rho/2,\rho)$ and any $a \in\R^N$, we have 
\begin{align*}
	& \sup_{t \in \Lambda_r (t_o)} 
 	\mint_{B_r^{(\theta)} (x_o)} 
	\frac{\big|\power{u}{\frac{1+m}{2}}(t) - \power{a}{\frac{1+m}{2}}\big|^2}	
	{r^{\frac{1+m}{m}}} \dx +
	\biint_{Q_r^{(\theta)}(z_o)} |D\power{u}{m}|^2 \dx\dt \\
	&\qquad\leq 
	c\,\biint_{Q_\rho^{(\theta)}(z_o)} 
	\bigg[ 
	\frac{\big|\power{u}{\frac{1+m}{2}}-\power{a}{\frac{1+m}{2}}\big|^2}
	{\rho^{\frac{1+m}{m}}-r^{\frac{1+m}{m}}} + 
	\frac{\big|\power{u}{m} - \power{a}{m}\big|^2}
	{\theta^{\frac{2m(m-1)}{1+m}}(\rho-r)^2} + 
	|F|^{2} \bigg]\dx\dt ,
\end{align*}
where $c=c(m,\nu,L)$.
\end{lemma}

The following lemma serves to compare the slice-wise mean values of a given weak solution at different times. It is often called {\em gluing lemma}. 
Such an assertion is necessary and very useful since Poincar\'e's and Sobolev's inequality can only be applied slice-wise. 
The proof  is exactly as in \cite[Lemma~3.2]{BDKS-higher-int}, taking into account the different definition of scaled cylinders.

\begin{lemma}\label{lem:time-diff}
Let $m>0$ and $u$ be a weak solution to \eqref{por-med-eq} in $\Omega_T$ in the sense of Definition~{\upshape\ref{def:weak_solution}}.
Then, for any cylinder $Q_{\rho}^{(\theta)}(z_o)\subseteq\Omega_T$ with
$\rho,\theta>0$ there exists $\hat\rho\in [\frac{\rho}{2},\rho]$ such that for all $t_1,t_2\in\Lambda_\rho(t_o)$ we have 
\begin{align*}
	\big|\langle u\rangle_{x_o;\hat\rho}^{(\theta)}(t_2) - 
	\langle u\rangle_{x_o;\hat\rho}^{(\theta)}(t_1)\big| 
	&\le
	c\,\theta^{\frac{m(1-m)}{1+m}}\rho^{\frac{1}{m}}
	\biint_{Q_{\rho}^{(\theta)}(z_o)} 
	\big[|D\power{u}{m}| + |F|\big] \dx\dt ,
\end{align*}
for a constant $c=c(L)$. 
\end{lemma}

\section{Sobolev-Poincar\'e type inequality}\label{sec:poin}

In this section we consider cylinders
$Q_\rho^{(\theta)}(z_o)\subseteq\Omega_T$, where $\rho,\theta>0$,
which satisfy a {\it sub-intrinsic coupling} in the sense that for some constant
$K\ge 1$ we have
\begin{equation}\label{sub-intrinsic-poincare}
  	\biint_{Q_\rho^{(\theta)}(z_o)} 
  	\frac{|u|^{1+m}}{\rho^{\frac{1+m}m}}\,\dx\dt
  	\le 
   	K \theta^{2m}.
\end{equation}
Furthermore, we assume that either 
\begin{equation}\label{super-intrinsic-poincare}
  	\theta^{2m}
  	\le 
  	K \biint_{Q_\rho^{(\theta)}(z_o)} 
  	\frac{|u|^{1+m}}{\rho^{\frac{1+m}m}}\,\dx\dt
  	\quad\mbox{or}\quad
  	\theta^{2m}
  	\le 
	K \biint_{Q_{\rho}^{(\theta)}(z_o)} 
	\big[|D\power{u}{m}|^2 + |F|^2\big] \dx\dt
\end{equation}
holds true. The principal goal of the section  is to establish the following Sobolev-Poincar\'e type inequality.
This inequality illustrates the significance of the lower bound $m>m_c$, since only in this case, we obtain an integrability exponent $2q<2$ on the right-hand side.

\begin{lemma}\label{lem:poin}
Let $m\in(m_c,1]$ and $u$ be a weak solution to \eqref{por-med-eq} in
$\Omega_T$ in the sense of
Definition~{\upshape\ref{def:weak_solution}}. Then, on any cylinder
$Q_\rho^{(\theta)}(z_o)\subseteq\Omega_T$ satisfying
\eqref{sub-intrinsic-poincare}, with $\rho,\theta>0$, and for any $\eps\in(0,1]$, we have 
\begin{align*}
	&\biint_{Q_\rho^{(\theta)}(z_o)} 
	\frac{\big|\power{u}{\frac{1+m}{2}}-
	(\power{u}{\frac{1+m}{2}})_{z_o;\rho}^{(\theta)}\big|^2}
	{\rho^{\frac{1+m}{m}}}  
	\,\dx\dt \\
	&\quad\le 
  	\epsilon\Bigg[
	\sup_{t\in\Lambda_\rho(t_o)} 
  	\bint_{B_\rho^{(\theta)}(x_o)} 
	\frac{\big|\power{u}{\frac{1+m}{2}}(t)-
	(\power{u}{\frac{1+m}{2}})_{z_o;\rho}^{(\theta)}\big|^2}
	{\rho^{\frac{1+m}{m}}} \,\dx +
	\biint_{Q^{(\theta)}_\rho(z_o)} |D\power um|^2 \,\dx\dt \Bigg]\\
	&\quad \phantom{\le\,}+
	\frac{c}{\epsilon^{\frac{2}{n}}}
	\Bigg[
  	\bigg[\biint_{Q_\rho^{(\theta)}(z_o)} 
	|D\power{u}{m}|^{2q}  \,\dx\dt 
  	\bigg]^{\frac{1}{q}} +
	\biint_{Q_\rho^{(\theta)}(z_o)} 
	|F|^{2}  \,\dx\dt\Bigg]
\end{align*}
for a constant $c=c(n,m,L,K)$. Here the integrability exponent $q$ is given by
\begin{equation}\label{def:q}
  q:=\max\bigg\{\frac{n(1+m)}{2(nm+1+m)},\frac12\bigg\}<1.
\end{equation}
\end{lemma}

\begin{proof}
Throughout  the proof we omit the center $z_o$ in our notation.  
By $\hat\rho\in [\frac12\rho, \rho]$ we denote the radius introduced in  Lemma \ref{lem:time-diff}.
We start our  considerations by estimating
\begin{align*}
	\biint_{Q_\rho^{(\theta)}} \!
	\frac{\big|\power{u}{\frac{1+m}{2}}-
	(\power{u}{\frac{1+m}{2}})_{\rho}^{(\theta)}\big|^2}
	{\rho^{\frac{1+m}{m}}}  
	\,\dx\dt
    &\le
	\biint_{Q_\rho^{(\theta)}} \!
	\frac{\big|\power{u}{\frac{1+m}{2}}-
	\power{\big[(u)_{\hat\rho}^{(\theta)}\big]}{\frac{1+m}{2}}\big|^2}
	{\rho^{\frac{1+m}{m}}}  
	\,\dx\dt 
    \le 
    2[\mathrm{I}+\mathrm{II}].
\end{align*}
Here we have abbreviated
\begin{align*}
	\mathrm{I}
    &:=
    \biint_{Q_\rho^{(\theta)}}
	\frac{\big|\power{u}{\frac{1+m}{2}}-
	\power{\big[\langle u\rangle_{\hat\rho}^{(\theta)}(t)\big]}{\frac{1+m}{2}}
	\big|^2}
    {\rho^{\frac{1+m}{m}}}\,\dx\dt,\\[7pt]
    \mathrm{II}
    &:=
    \bint_{\Lambda_\rho}
    \frac{\big|\power{\big[\langle u\rangle_{\hat\rho}^{(\theta)}(t)\big]}{\frac{1+m}{2}}-
    \power{\big[(u)_{\hat\rho}^{(\theta)}\big]}{\frac{1+m}{2}}\big|^2}
    {\rho^{\frac{1+m}{m}}}\,\dt.
\end{align*}
The first term can be estimated with Young's inequality and Lemma~\ref{lem:alphalemma}. We obtain
\begin{align}\label{estimate-of-I}
    \mathrm{I}
    &\le
    \frac{1}{\rho^{\frac{1+m}m}}
    \sup_{t\in\Lambda_\rho}
    \bigg[\bint_{B_\rho^{(\theta)}} 
    \big|\power{u}{\frac{1+m}{2}}-
	\power{\big[\langle u\rangle_{\hat\rho}^{(\theta)}(t)\big]}{\frac{1+m}{2}}
	\big|^2 \,\dx \bigg]^{\frac{2}{n+2}} \nonumber\\
	&\qquad\cdot
	\bint_{\Lambda_\rho}
    \bigg[\bint_{B_\rho^{(\theta)}} 
    \big|\power{u}{\frac{1+m}{2}}-
	\power{\big[\langle u\rangle_{\hat\rho}^{(\theta)}(t)\big]}{\frac{1+m}{2}}
	\big|^2 \,\dx
    \bigg]^{\frac{n}{n+2}} \dt \nonumber\\
    &\le
    \eps \sup_{t\in\Lambda_\rho}
    \bint_{B_\rho^{(\theta)}}
    \frac{\big|\power{u}{\frac{1+m}{2}}(t)-
	(\power{u}{\frac{1+m}{2}})_{\rho}^{(\theta)}\big|^2}
	{\rho^{\frac{1+m}m}} \,\dx
    +
    \frac{c}{\eps^{\frac2n}\rho^{\frac{1+m}m}} \,
    \mathrm{III}^{\frac{n+2}{n}} ,
\end{align}
where $c=c(n,m)$ and 
$$
	\mathrm{III}
	:=
	\bint_{\Lambda_\rho}
    \bigg[\bint_{B_\rho^{(\theta)}} 
    \big|\power{u}{\frac{1+m}{2}}-
	\power{\big[\langle u^m\rangle_{\rho}^{(\theta)}(t)\big]}{\frac{1+m}{2m}}
	\big|^2 \,\dx
    \bigg]^{\frac{n}{n+2}} \dt.
$$
If $m<1$ we estimate the integral $\mathrm{III}$ by means of Lemma~\ref{lem:Acerbi-Fusco} with $\alpha=\frac{1+m}{2m}$ and H\"older's inequality  in
 space with exponents $\frac{1+m}{1-m}$ and $\frac{1+m}{2m}$, which yields
\begin{align*}
	\mathrm{III}
    &\le
	\bint_{\Lambda_\rho}
    \bigg[\bint_{B_\rho^{(\theta)}} 
    \big(|\power{u}{m}|+|\langle \power{u}{m}\rangle_\rho^{(\theta)}(t)|\big)^{\frac{1-m}{m}}
    \big|\power{u}{m}-\langle\power{u}{m}\rangle_\rho^{(\theta)}(t)\big|^2 \,\dx
    \bigg]^{\frac{n}{n+2}} \dt \\
    &\le
    \bint_{\Lambda_\rho}
    \bigg[\bint_{B_\rho^{(\theta)}}
    |u|^{1+m}
    \,\dx\bigg]^{\frac{n}{n+2}\frac{1-m}{1+m}}
    \bigg[\bint_{B_\rho^{(\theta)}}
    \big|\power{u}{m}-\langle\power{u}{m}\rangle_\rho^{(\theta)}(t)\big|^{\frac{1+m}{m}} \,\dx
	\bigg]^{\frac{n}{n+2}\frac{2m}{1+m}} \dt.
\end{align*}
To proceed further, we recall the definition of  $q$. Now, again in
the case $m<1$, we apply H\"older's inequality in time with exponents $\frac{(n+2)(1+m)}{n(1-m)}$ and $\frac{(n+2)(1+m)}{2(nm+1+m)}\le\frac{q(n+2)}{n}$, the sub-intrinsic coupling \eqref{sub-intrinsic-poincare} and Sobolev's inequality on the time slices. This leads to
\begin{align*}
	\mathrm{III}
    &\le
    \bigg[
    \biint_{Q_\rho^{(\theta)}}|u|^{1+m} \,\dx\dt
    \bigg]^{\frac{n(1-m)}{(n+2)(1+m)}}\\
	&\qquad\qquad\cdot
    \Bigg[\bint_{\Lambda_\rho}
    \bigg[\bint_{B_\rho^{(\theta)}}
    \big|\power{u}{m}-\langle\power{u}{m}\rangle_\rho^{(\theta)}(t)\big|^{\frac{1+m}{m}} \,\dx
    \bigg]^{\frac{2mq}{1+m}}\dt\Bigg]^{\frac{n}{q(n+2)}}\\
    &\le
    c\,\big(\theta^{2m}\rho^{\frac{1+m}{m}}\big)^{\frac{n(1-m)}{(n+2)(1+m)}}
    \big(\theta^{\frac{m(m-1)}{1+m}}\rho\big)^{\frac{2n}{n+2}}
    \bigg[\biint_{Q_\rho^{(\theta)}}
	|D\power{u}{m}|^{2q} \,\dx\dt
	\bigg]^{\frac{n}{q(n+2)}}\\
    &=
    c\Bigg[\rho^{\frac{1+m}{m}}
    \bigg[\biint_{Q_\rho^{(\theta)}}|D\power{u}{m}|^{2q}
    \,\dx\dt\bigg]^{\frac{1}{q}} \Bigg]^{\frac{n}{n+2}},
\end{align*}
where $c=c(n,m,K)$.   
Note that this inequality also holds true for $m=1$. In this case we directly  apply Sobolev's inequality on the time slices.  In any case, 
the combination of the last inequality with \eqref{estimate-of-I}
yields 
\begin{align*}
	\mathrm{I}
    \le
    \eps \sup_{t\in\Lambda_\rho}
    \bint_{B_\rho^{(\theta)}}
    \frac{\big|\power{u}{\frac{1+m}{2}}(t)-
	(\power{u}{\frac{1+m}{2}})_{\rho}^{(\theta)}\big|^2}
	{\rho^{\frac{1+m}m}}\,\dx
  	+
  	\frac{c}{\epsilon^{\frac{2}{n}}}
  	\bigg[\biint_{Q_\rho^{(\theta)}} 
  	|D\power{u}{m}|^{2q} \,\dx\dt 
  	\bigg]^{\frac{1}{q}}.
\end{align*}
It remains to estimate $\mathrm{II}$. To this end,
we use the fact $\frac{1+m}{2}\le 1$ in Lemma \ref{lem:a-b} and the gluing Lemma \ref{lem:time-diff} to deduce 
\begin{align}\label{est-II-1}
     \mathrm{II}
     &\le
     c\,
     \bint_{\Lambda_\rho}
     \frac{\big|\langle u\rangle_{\hat\rho}^{(\theta)}(t)-(u)_{\hat\rho}^{(\theta)}\big|^{1+m}}
     {\rho^{\frac{1+m}{m}}}\,\dt \nonumber\\
     &\le
     \frac{c}{\rho^{\frac{1+m}m}}
     \bint_{\Lambda_\rho}\bint_{\Lambda_\rho}
     \big|\langle u\rangle_{\hat\rho}^{(\theta)}(t)-\langle u\rangle_{\hat\rho}^{(\theta)}(\tau)\big|^{1+m}
     \,\dt\d\tau \nonumber\\
     &\le
     c\,\theta^{m(1-m)}
     \bigg[\biint_{Q_\rho^{(\theta)}}\big[|D\power um|+|F|\big] \dx\dt
     \bigg]^{1+m}
\end{align}
for a constant $c=c(m,L)$.
If either~\eqref{super-intrinsic-poincare}$_2$ is satisfied or if $m=1$, then we have 
\begin{align*}
     \mathrm{II}
     &\le
     c\,\bigg[\biint_{Q_\rho^{(\theta)}}\big[|D\power um|^2+|F|^2\big] \dx\dt
     \bigg]^{\frac{1-m}{2}}
     \bigg[\biint_{Q_\rho^{(\theta)}}\big[|D\power um|+|F|\big] \dx\dt
     \bigg]^{1+m} \\
     &\le
     \epsilon\,\biint_{Q_\rho^{(\theta)}}|D\power um|^2 \dx\dt +
     \frac{c}{\epsilon^{\frac{1-m}{1+m}}}\Bigg[
     \bigg[\biint_{Q_\rho^{(\theta)}}|D\power um|^{2q} \dx\dt
     \bigg]^{\frac1q} +
     \biint_{Q_\rho^{(\theta)}} |F|^2 \dx\dt\Bigg],
\end{align*}
where the constant $c$ depends only on $m$, $L$ and  $K$.
Together with the estimate for $\mathrm{I}$, this proves the asserted
inequality. Note that $\frac{1-m}{1+m}<\frac{2}{n}$ since $m>m_c$. Otherwise, if $m<1$ and~\eqref{super-intrinsic-poincare}$_1$ is satisfied, then we argue as follows. First, observe that
\begin{align*}
     \theta^{2m}
     &\le
     2K\,\biint_{Q_\rho^{(\theta)}}
     \frac{\big|\power{u}{\frac{1+m}{2}}-
     \power{\big[(u)_{\hat\rho}^{(\theta)}\big]}{\frac{1+m}{2}}\big|^2}
     {\rho^{\frac{1+m}m}}\,\dx\dt
     +
     \frac{2K\,\big|(u)_{\hat\rho}^{(\theta)}\big|^{1+m}}{\rho^{\frac{1+m}{m}}} \,.
\end{align*}
Therefore, we have
\begin{align*}
     \mathrm{II}
     &=
     \frac{\theta^{\frac{2m(1-m)}{1+m}}\, \mathrm{II}}{\theta^{\frac{2m(1-m)}{1+m}}}
     \le
     c\big[\mathrm{II}_1+\mathrm{II}_2\big],
\end{align*}
with
\begin{align*}
     \mathrm{II}_1
     &:=
     \frac{1}{\theta^{\frac{2m(1-m)}{1+m}}}
     \Bigg[\biint_{Q_\rho^{(\theta)}}
     \frac{\big|\power{u}{\frac{1+m}{2}}-
     \power{\big[(u)_{\hat\rho}^{(\theta)}\big]}{\frac{1+m}{2}}\big|^2}
     {\rho^{\frac{1+m}m}}\,\dx\dt
     \Bigg]^{\frac{1-m}{1+m}}\cdot\mathrm{II}, \\[7pt]
     \mathrm{II}_2
     &:=
     \frac{\big|(u)_{\hat\rho}^{(\theta)}\big|^{1-m}}{\theta^{\frac{2m(1-m)}{1+m}}\rho^{\frac{1-m}{m}}}\cdot\mathrm{II}.
\end{align*}
To estimate $\mathrm{II}_1$, we apply in turn~\eqref{est-II-1}, assumption \eqref{sub-intrinsic-poincare}, Lemma~\ref{lem:alphalemma} and
Young's inequality with exponents $\frac{2}{1-m}$, $\frac{2}{1+m}$. This gives
\begin{align*}
	\mathrm{II}_1
    &\le
    \frac{c}{\theta^{\frac{m(1-m)^2}{1+m}}}
    \Bigg[\biint_{Q_\rho^{(\theta)}}
    \frac{\big|\power{u}{\frac{1+m}{2}}-
    \power{\big[(u)_{\hat\rho}^{(\theta)}\big]}{\frac{1+m}{2}}\big|^2}
    {\rho^{\frac{1+m}m}}\,\dx\dt
    \Bigg]^{\frac{1-m}{1+m}}\\
    &\qquad\qquad\qquad\qquad\cdot
    \bigg[
    \biint_{Q_\rho^{(\theta)}}\big[|D\power um|+|F|\big]\dx\dt
    \bigg]^{1+m}\\
    &\le
    c\Bigg[\biint_{Q_\rho^{(\theta)}}
    \frac{\big|\power{u}{\frac{1+m}{2}}-
    \power{\big[(u)_{\hat\rho}^{(\theta)}\big]}{\frac{1+m}{2}}\big|^2}
    {\rho^{\frac{1+m}m}}\,\dx\dt
    \Bigg]^{\frac{1-m}{2}}
    \bigg[
    \biint_{Q_\rho^{(\theta)}}\big[|D\power um|+|F|\big]\dx\dt
    \bigg]^{1+m}\\
    &\le
    \tfrac12
    \biint_{Q_\rho^{(\theta)}}
    \frac{\big|\power{u}{\frac{1+m}{2}}-
    (\power{u}{\frac{1+m}{2}})_{\rho}^{(\theta)}\big|^2}
    {\rho^{\frac{1+m}{m}}}\,\dx\dt
    +
    c\bigg[\biint_{Q_\rho^{(\theta)}}\big[|D\power um|+|F|\big]\d x\d t
    \bigg]^2,
\end{align*}
with a constant $c=c(m,L,K)$.
For the term $\mathrm{II}_2$, we proceed as follows. We first insert the expression
for the term $\mathrm{II}$, then use Lemma \ref{lem:Acerbi-Fusco} with $\alpha:=\frac{2}{1+m}$, and
finally apply the gluing Lemma~\ref{lem:time-diff}. This leads to
 \begin{align*}
     \mathrm{II}_2
     &\le
     \frac{c}{\theta^{\frac{2m(1-m)}{1+m}}\rho^{\frac2m}}\,
     \bint_{\Lambda_\rho}
     \big|(u)_{\hat\rho}^{(\theta)}\big|^{1-m}
     \Big|\power{\big[\langle u\rangle_{\hat\rho}^{(\theta)}(t)\big]}{\frac{1+m}{2}}-
     \power{\big[(u)_{\hat\rho}^{(\theta)}\big]}{\frac{1+m}{2}}\Big|^2
     \,\dt\\
     &\le
     \frac{c}{\theta^{\frac{2m(1-m)}{1+m}}\rho^{\frac2m}}\,
     \bint_{\Lambda_\rho}
     \big|\langle u\rangle_{\hat\rho}^{(\theta)}(t)-(u)_{\hat\rho}^{(\theta)}\big|^2
     \,\dt\\
     &\le
     \frac{c}{\theta^{\frac{2m(1-m)}{1+m}}\rho^{\frac2m}}\,
     \bint_{\Lambda_\rho}\bint_{\Lambda_\rho}
     \big|\langle u\rangle_{\hat\rho}^{(\theta)}(t)-\langle u\rangle_{\hat\rho}^{(\theta)}(\tau)\big|^2
     \,\dt\d\tau\\
     &\le
     c\bigg[
     \biint_{Q^{(\theta)}_\rho} \big[|D\power um|+|F|\big] \dx\dt\bigg]^2,
\end{align*}
again with  a constant $c=c(m,L,K)$.
Collecting the estimates for $\mathrm{I}$, $\mathrm{II}_1$, and $\mathrm{II}_2$, we arrive at
\begin{align*}
	\biint_{Q_\rho^{(\theta)}} &
    \frac{\big|\power{u}{\frac{1+m}{2}}-
    (\power{u}{\frac{1+m}{2}})_{\rho}^{(\theta)}\big|^2}
    {\rho^{\frac{1+m}{m}}} \,\dx\dt \\
    &\le
    \epsilon\sup_{t\in\Lambda_\rho} 
 	\bint_{B_\rho^{(\theta)}} \frac{\big|\power{u}{\frac{1+m}{2}}(t)-
    (\power{u}{\frac{1+m}{2}})_{\rho}^{(\theta)}\big|^2}
    {\rho^{\frac{1+m}{m}}} \dx  
  	+
  	\frac{c}{\epsilon^{\frac{2}{n}}}
  	\bigg[\biint_{Q_\rho^{(\theta)}} 
  	|D\power{u}{m}|^{2q} \dx \dt 
  	\bigg]^{\frac{1}{q}}\\
  	&\phantom{\le\,} +
    \tfrac12 \biint_{Q_\rho^{(\theta)}}
	\frac{\big|\power{u}{\frac{1+m}{2}}-
    (\power{u}{\frac{1+m}{2}})_{\rho}^{(\theta)}\big|^2}
    {\rho^{\frac{1+m}{m}}} \,\dx\dt +
    c\bigg[
    \biint_{Q^{(\theta)}_\rho} \big[|D\power um|+|F|\big] \dx\dt
	\bigg]^2.
\end{align*}
Re-absorbing the second last term into the left-hand side, and applying in turn H\"older's inequality
we again obtain the asserted Sobolev-Poincar\'e inequality. 
\end{proof}

\section{Reverse H\"older inequality}\label{sec:revholder}

The core of any proof of higher integrability of the gradient is a reverse H\"older inequality.
In this section we establish such an inequality on certain  intrinsic cylinders. 
Throughout this section we assume that $Q_{2\rho}^{(\theta)}(z_o)\subseteq\Omega_T$ with $\rho,\theta> 0$ is a scaled cylinder satisfying a sub-intrinsic coupling 
\begin{equation}\label{sub-intrinsic}
  	\biint_{Q_{2\rho}^{(\theta)}(z_o)} 
  	\frac{|u|^{1+m}}{(2\rho)^{\frac{1+m}m}}\,\dx\dt
  	\le 
   	K \theta^{2m},
\end{equation}
for some constant $K\ge 1$. Furthermore, we assume that either 
\begin{equation}\label{super-intrinsic}
  	\theta^{2m}
  	\le 
  	K \biint_{Q_\rho^{(\theta)}(z_o)} \!
  	\frac{|u|^{1+m}}{\rho^{\frac{1+m}m}}\,\dx\dt
  	\quad\mbox{or}\quad
  	\theta^{2m}
  	\le 
	K \biint_{Q_{\rho}^{(\theta)}(z_o)} \!
	\big[|D\power{u}{m}|^2 + |F|^2\big] \dx\dt.
\end{equation}
This specifies the setup for the following reverse H\"older inequality.

\begin{proposition}\label{prop:revhoelder}
Let $m\in(m_c,1]$ and $u$ be a weak solution to \eqref{por-med-eq} in $\Omega_T$ in the sense of Definition~{\upshape\ref{def:weak_solution}}. 
Then, on any  cylinder $Q_{2\rho}^{(\theta)}(z_o)\subseteq\Omega_T$ with $\rho,\theta>0$ satisfying \eqref{sub-intrinsic} and \eqref{super-intrinsic}, we have
\begin{align*}
	\biint_{Q_{\rho}^{(\theta)}(z_o)} |D\power{u}{m}|^2 \dx\dt 
	\le
	c\bigg[\biint_{Q_{2\rho}^{(\theta)}(z_o)} 
	|D\power{u}{m}|^{2q} \dx\dt \bigg]^{\frac{1}{q}} +
	c\, \biint_{Q_{2\rho}^{(\theta)}(z_o)} |F|^{2} \dx\dt ,
\end{align*}
for a constant $c=c(n,m,\nu, L,K)$. Here, $q<1$ 
is the integrability exponent from \eqref{def:q}.

\end{proposition}

\begin{proof} 
We omit the reference to the center $z_o$ in the notation and consider radii $r,s$ with $\rho\le r<s\le 2\rho$. 
Note that hypothesis \eqref{sub-intrinsic} and \eqref{super-intrinsic} imply that the coupling conditions
\eqref{sub-intrinsic-poincare} and \eqref{super-intrinsic-poincare} are
satisfied on $Q_s^{(\theta)}$ with constant  $2^{n+2+\frac2m}K$
instead of $K$. 
From the energy estimate in Lemma \ref{lem:energy}, we obtain  with a constant $c=c(m,\nu,L)$ that
\begin{align*}
	&\sup_{t \in \Lambda_r}
	\mint_{B_r^{(\theta)}} 
	\frac{\big|\power{u}{\frac{1+m}{2}}(t) - 
	(\power{u}{\frac{1+m}{2}})_r^{(\theta)}\big|^2}{r^{\frac{1+m}{m}}} \dx +
	\biint_{Q_r^{(\theta)}} |D\power{u}{m}|^2 \dx\dt \nonumber\\
	&\quad\le
	c\,\biint_{Q_s^{(\theta)}}
	\frac{\big|\power{u}{\frac{1+m}{2}} - 
	(\power{u}{\frac{1+m}{2}})_r^{(\theta)}\big|^2}
	{s^{\frac{1+m}{m}}-r^{\frac{1+m}{m}}}  \dx\dt +
	c\,\biint_{Q_s^{(\theta)}} 
	\frac{\big|\power{u}{m} - 
	\power{\big[(\power{u}{\frac{1+m}{2}})_r^{(\theta)}\big]}{\frac{2m}{1+m}}
	\big|^2}
	{\theta^{\frac{2m(m-1)}{1+m}}(s-r)^2} \dx\dt
	\nonumber\\
	&\quad\quad +
	c\, \biint_{Q_{s}^{(\theta)}} |F|^{2}\dx\dt \nonumber\\
	&\quad =:
	\mbox{I} + \mbox{II} + \mbox{III},
\end{align*}
where the meaning of $\mathrm I$, $\mathrm{II}$ and $\mathrm{III}$ is clear in this context. We let
\begin{equation*}
  \mathcal R_{r,s}
  :=
  \frac{s}{s-r}.
\end{equation*}
To estimate the term $\mathrm I$ we first observe that $(s-r)^{\frac{1+m}{m}} \le s^{\frac{1+m}{m}}-r^{\frac{1+m}{m}}$. This, together with an application of Lemma~\ref{lem:alphalemma} implies
\begin{align*}
	\mbox{I}
	&\le
	c\,\mathcal R_{r,s}^{\frac{1+m}m}
	\biint_{Q_s^{(\theta)}} 
	\frac{\big|\power{u}{\frac{1+m}{2}} - 
	(\power{u}{\frac{1+m}{2}})_s^{(\theta)}\big|^2}
	{s^{\frac{1+m}{m}}}  \dx\dt ,
\end{align*}
again with a constant $c$ depending on $m,\nu ,L$ only.
We now turn our attention to the term $\mbox{II}$, which we re-write as
\begin{align*}
	\mbox{II}
	=
	c\,\mathcal R_{r,s}^{2} \theta^{\frac{2m(1-m)}{1+m}}
	\biint_{Q_s^{(\theta)}} 
	\frac{\big|\power{u}{m} - 
	\power{\big[(u^{\frac{1+m}{2}})_r^{(\theta)}\big]}{\frac{2m}{1+m}}
	\big|^2}
	{s^2} \dx\dt .
\end{align*}
If \eqref{super-intrinsic}$_2$ is satisfied, we apply Lemma~\ref{lem:a-b}, H\"older's inequality, Lemma~\ref{lem:alphalemma} and Young's inequality to obtain for $\epsilon\in(0,1]$ that
\begin{align*}
	\mbox{II}
	&\le 
	c\mathcal R_{r,s}^{2} \theta^{\frac{2m(1-m)}{1+m}}
	\Bigg[\biint_{Q_s^{(\theta)}} 
	\frac{\big|\power{u}{\frac{1+m}{2}} - 
	(\power{u}{\frac{1+m}{2}})_r^{(\theta)}\big|^2}
	{s^{\frac{1+m}{m}}} \dx\dt\Bigg]^{\frac{2m}{1+m}} \\
	&\le
	c\mathcal R_{r,s}^{2} \bigg[
	\biint_{Q_{\rho}^{(\theta)}}\! \!
	\big[|D\power{u}{m}|^2 + |F|^2\big] \dx\dt \bigg]^{\frac{1-m}{1+m}}
	\Bigg[\biint_{Q_s^{(\theta)}}\! \!
	\frac{\big|\power{u}{\frac{1+m}{2}} - 
	(\power{u}{\frac{1+m}{2}})_s^{(\theta)}\big|^2}
	{s^{\frac{1+m}{m}}} \dx\dt\Bigg]^{\frac{2m}{1+m}} \\
	&\le
	\epsilon\mathcal R_{r,s}^{2}
	\biint_{Q_{s}^{(\theta)}} 
	\big[|D\power{u}{m}|^2 + |F|^2\big] \dx\dt +
	\frac{c\,\mathcal R_{r,s}^{2}}{\epsilon^{\frac{1-m}{2m}}}
	\biint_{Q_s^{(\theta)}} 
	\frac{\big|\power{u}{\frac{1+m}{2}} - 
	(\power{u}{\frac{1+m}{2}})_s^{(\theta)}\big|^2}
	{s^{\frac{1+m}{m}}} \dx\dt 
\end{align*}
with $c=c(m,\nu,L,K)$. In the case $m=1$, this estimate follows even
without an application of Young's inequality.
Otherwise, if~\eqref{super-intrinsic}$_1$ is in force, we have that
\begin{align*}
  	\theta^{2m}
  	&\le
  	2^{n+2+\frac2m}K 
	\biint_{Q_s^{(\theta)}} 
  	\frac{|u|^{1+m}}{s^{\frac{1+m}m}}\,\dx\dt \\
	&\le
	c\, 
  	\biint_{Q_s^{(\theta)}} 
  	\frac{\big|\power{u}{\frac{1+m}{2}}-
  	(\power{u}{\frac{1+m}{2}})_{r}^{(\theta)}\big|^{2}}
  	{s^{\frac{1+m}m}}\,\dx\dt +
  	\frac{c\,\big|(\power{u}{\frac{1+m}{2}})_r^{(\theta)}\big|^2}
  	{s^{\frac{1+m}m}} .
\end{align*}
This leads to
\begin{align*}
	\mbox{II} 
	\le
	c\,\mathcal R_{r,s}^{2} [\mathrm{II}_1 + \mathrm{II}_2],
\end{align*}
where we have set
\begin{equation*}
	\mathrm{II}_1
	:=
	\Bigg[\biint_{Q_s^{(\theta)}}\!\!\!
  	\frac{\big|\power{u}{\frac{1+m}{2}}-
  	(\power{u}{\frac{1+m}{2}})_{r}^{(\theta)}\big|^{2}}
  	{s^{\frac{1+m}m}} \dx\dt\Bigg]^{\frac{1-m}{1+m}} 
	\biint_{Q_s^{(\theta)}} \!\!\!
	\frac{\big|\power{u}{m} - 
	\power{\big[(u^{\frac{1+m}{2}})_r^{(\theta)}\big]}{\frac{2m}{1+m}}\big|^2}
	{s^2} \dx\dt
\end{equation*}
and
\begin{align*}
	\mathrm{II}_2
	&:=
	\frac{\big|(\power{u}{\frac{1+m}{2}})_r^{(\theta)}\big|^{\frac{2(1-m)}{1+m}}}
	{s^{\frac{1-m}{m}}} 
	\biint_{Q_s^{(\theta)}} \!\!\!
	\frac{\big|\power{u}{m} - 
	\power{\big[(u^{\frac{1+m}{2}})_r^{(\theta)}\big]}{\frac{2m}{1+m}}\big|^2}
	{s^2} \dx\dt .
\end{align*}
To term $\mathrm{II}_1$ we apply in turn  Lemma~\ref{lem:a-b}, H\"older's inequality and Lemma~\ref{lem:alphalemma}, and obtain
\begin{align*}
	\mathrm{II}_1
	\le
	c\, \biint_{Q_s^{(\theta)}} 
  	\frac{\big|\power{u}{\frac{1+m}{2}}-
  	(\power{u}{\frac{1+m}{2}})_{r}^{(\theta)}\big|^{2}}
  	{s^{\frac{1+m}m}}\,\dx\dt
	\le
	c\, \biint_{Q_s^{(\theta)}} 
  	\frac{\big|\power{u}{\frac{1+m}{2}}-
  	(\power{u}{\frac{1+m}{2}})_{s}^{(\theta)}\big|^{2}}
  	{s^{\frac{1+m}m}}\,\dx\dt,
\end{align*}
while to term $\mathrm{II}_2$ we apply  Lemma~\ref{lem:Acerbi-Fusco} with $\alpha =\frac{1+m}{2m}$ and Lemma~\ref{lem:alphalemma} and find
\begin{align*}
	\mathrm{II}_2
	\le
	c\, \biint_{Q_s^{(\theta)}} 
	\frac{\big|\power{u}{\frac{1+m}{2}} - 
	(\power{u}{\frac{1+m}{2}})_r^{(\theta)}\big|^2}
	{s^{\frac{1+m}{m}}} \dx\dt 
	\le
	c\, \biint_{Q_s^{(\theta)}} 
	\frac{\big|\power{u}{\frac{1+m}{2}} - 
	(\power{u}{\frac{1+m}{2}})_s^{(\theta)}\big|^2}
	{s^{\frac{1+m}{m}}} \dx\dt .
\end{align*}
Combining both cases we have
\begin{align*}
	\mbox{II}
	&\le
	\epsilon\,\mathcal R_{r,s}^{2}
	\biint_{Q_{s}^{(\theta)}} 
	\big[|D\power{u}{m}|^2 + |F|^2\big] \dx\dt +
	\frac{c\,\mathcal R_{r,s}^{2}}{\epsilon^{\frac{1-m}{2m}}}
	\biint_{Q_s^{(\theta)}} 
	\frac{\big|\power{u}{\frac{1+m}{2}}-
	(\power{u}{\frac{1+m}{2}})_{s}^{(\theta)}\big|^2}{s^{\frac{1+m}{m}}}  
	\dx\dt,
\end{align*}
with a constant $c=c(m,\nu,L,K)$. 
Inserting the estimates for I and II above and applying Lemma~\ref{lem:poin} with $\epsilon$ replaced by $\epsilon^{\frac{1+m}{2m}}$, we find for any $\epsilon\in(0,1]$ that 
\begin{align*}
	\sup_{t \in \Lambda_r}&
	\mint_{B_r^{(\theta)}} 
	\frac{\big|\power{u}{\frac{1+m}{2}}(t) - 
	(\power{u}{\frac{1+m}{2}})_r^{(\theta)}\big|^2}{r^{\frac{1+m}{m}}} \dx +
      \biint_{Q_r^{(\theta)}} |D\power{u}{m}|^2 \dx\dt \nonumber\\
      &\le
	\frac{c\,\mathcal R_{r,s}^{\frac{1+m}{m}}}{\epsilon^{\frac{1-m}{2m}}}
	\biint_{Q_s^{(\theta)}} 
	\frac{\big|\power{u}{\frac{1+m}{2}}-
	(\power{u}{\frac{1+m}{2}})_{s}^{(\theta)}\big|^2}{s^{\frac{1+m}{m}}}  
        \dx\dt\\
       &\qquad +
        \epsilon\,\mathcal R_{r,s}^{2}
	\biint_{Q_{s}^{(\theta)}} 
	|D\power{u}{m}|^2 \dx\dt
        +
        c\mathcal R_{r,s}^{2}
	\biint_{Q_{s}^{(\theta)}} 
	|F|^2\dx\dt
	\\
	&\le
	c\,\epsilon\mathcal R_{r,s}^{\frac{1+m}{m}}
	\Bigg[
	\sup_{t \in \Lambda_s}
	\mint_{B_s^{(\theta)}} 
	\frac{\big|\power{u}{\frac{1+m}{2}}(t) - 
	(\power{u}{\frac{1+m}{2}})_s^{(\theta)}\big|^2}{s^{\frac{1+m}{m}}} \dx +
	\biint_{Q_s^{(\theta)}} |D\power{u}{m}|^2 \dx\dt\Bigg]\\
	&\qquad +
        \frac{c\,\mathcal R_{r,s}^{\frac{1+m}{m}}}
	{\epsilon^{\frac{(1+m)(n+2)}{2nm}-1}}
	\Bigg[
	\bigg[\biint_{Q_s^{(\theta)}} 
	|D\power{u}{m}|^{2q} \dx \dt 
	\bigg]^{\frac{1}{q}}
	 +
	\biint_{Q^{(\theta)}_s} |F|^2 \dx\dt \Bigg].
\end{align*}
Here we choose  $\epsilon=1/[2c\mathcal R_{r,s}^{\frac{1+m}{m}}]$. With this choice
the first term on the right-hand side tuns into $\frac12[\dots ]$, where $[\dots ]$ is the expression
from the left-hand side with $r$ replaced by $s$. Moreover, the pre-factor in front of the second term on the right-hand side changes to $\mathcal R_{r,s}^{\alpha}$ with $\alpha=\frac{n+2}{2n}(\frac{1+m}m)^2$. To this inequality we apply the Iteration Lemma \ref{lem:tech} to  re-absorb the term $\frac12[\dots]$ (with radius $s$) from the right-hand side into the left. This leads to the claimed reverse H\"older type inequality
\begin{align*}
	\sup_{t \in \Lambda_\rho}&
	\mint_{B_\rho^{(\theta)}} 
	\frac{\big|\power{u}{\frac{1+m}{2}}(t) - 
	(\power{u}{\frac{1+m}{2}})_\rho^{(\theta)}\big|^2}
	{\rho^{\frac{1+m}{m}}} \dx +
	\biint_{Q_\rho^{(\theta)}} |D\power{u}{m}|^2 \dx\dt \\
	&\le
	c\bigg[\biint_{Q_{2\rho}^{(\theta)}} 
	|D\power{u}{m}|^{2q} \dx\dt \bigg]^{\frac{1}{q}} +
	c\,\biint_{Q_{2\rho}^{(\theta)}} |F|^{2} \dx\dt ,
\end{align*}
and finishes the proof.
\end{proof}

At the end of this section we provide a technical auxiliary result, which essentially is a direct consequence
of Lemma~\ref{lem:poin} and the energy estimate.

\begin{lemma}\label{lem:theta}
Let $m\in(m_c,1]$ and $u$ be a weak solution to \eqref{por-med-eq} in $\Omega_T$ in the sense of Definition~{\upshape\ref{def:weak_solution}}.
Then, on any  cylinder $Q_{2\rho}^{(\theta)}(z_o)\subseteq\Omega_T$ with $\rho,\theta>0$ satisfying \eqref{sub-intrinsic} and \eqref{super-intrinsic}$_1$ with $K=1$, we have
\begin{align*}
  	\theta^m
  	\le
          \tfrac1{\sqrt2}
        \Bigg[\biint_{Q_{\rho/2}^{(\theta)}(z_o)}
    \frac{|u|^{1+m}}{(\rho/2)^{\frac{1+m}{m}}} \,\dx\dt 
    \Bigg]^{\frac12} +
  	c\,\bigg[\biint_{Q_{2\rho}^{(\theta)}(z_o)} 
   	\big[|D\power{u}{m}|^{2}+|F|^2\big] \dx\dt\bigg]^{\frac12} ,
\end{align*}
where $c=c(n,m,\nu,L)$.
\end{lemma}

\begin{proof}
We omit the reference to the center $z_o$ in the notation. 
We use \eqref{super-intrinsic}$_1$ with $K=1$, Minkowski's inequality and Lemma~\ref{lem:alphalemma} to deduce
\begin{align*}
    \theta^{m}
    &\le
    \bigg[\biint_{Q_{\rho}^{(\theta)}} 
    \frac{|u|^{1+m}}{\rho^{\frac{1+m}m}} \,\dx\dt\bigg]^{\frac12} \\
    &\le 
    \Bigg[\biint_{Q_{\rho}^{(\theta)}}  
    \frac{\big|\power{u}{\frac{1+m}{2}} -
    (\power{u}{\frac{1+m}{2}})_{\rho/2}^{(\theta)}\big|^{2}}
    {\rho^{\frac{1+m}{m}}}
    \,\dx\dt
    \Bigg]^{\frac12} +
    \frac{\babs{(\power{u}{\frac{1+m}{2}})_{\rho/2}^{(\theta)}}}
    {\rho^{\frac{1+m}{2m}}} \\
    &\le 
    c\,\Bigg[\biint_{Q_{\rho}^{(\theta)}}  
    \frac{\big|\power{u}{\frac{1+m}{2}} -
    (\power{u}{\frac{1+m}{2}})_{\rho}^{(\theta)}\big|^{2}}
    {\rho^{\frac{1+m}{m}}}
    \,\dx\dt
    \Bigg]^{\frac12} +
    \bigg[\biint_{Q_{\rho/2}^{(\theta)}}
    \frac{|u|^{1+m}}{\rho^{\frac{1+m}{m}}} \,\dx\dt 
    \bigg]^{\frac12}.
\end{align*}
We estimate the first term on the right with Lemma~\ref{lem:poin} and H\"older's inequality and get 
\begin{align*}
    \biint_{Q_{\rho}^{(\theta)}} 
  	\frac{\big|\power{u}{\frac{1+m}{2}} - 
  	(\power{u}{\frac{1+m}{2}})_{\rho}^{(\theta)}\big|^{2}}
	{\rho^{\frac{1+m}{m}}} \,\dx\dt 
	&\le
   	\epsilon
   	\sup_{t\in\Lambda_{\rho}} 
   	\bint_{B_{\rho}^{(\theta)}} 
   	\frac{\big|\power{u}{\frac{1+m}{2}}(t)-
   	(\power{u}{\frac{1+m}{2}})_{\rho}^{(\theta)}\big|^2}
   	{\rho^{\frac{1+m}{m}}} \,\dx \\
   	&\quad+
   	\frac{c}{\epsilon^{\frac{2}{n}}}
   	\biint_{Q_{\rho}^{(\theta)}} 
   	\big[|D\power{u}{m}|^{2}+|F|^2\big]  \dx\dt , 
\end{align*}
for a constant $c=c(n,m,L)$ and an arbitrary $\epsilon\in(0,1]$. 
In order to bound the $\sup$-term appearing in the last estimate, we
apply the energy estimate from Lemma \ref{lem:energy},
 combined with Lemma \ref{lem:alphalemma} with $a=0$, H\"older's inequality and hypothesis \eqref{sub-intrinsic},
 with the result
\begin{align*}
    &\sup_{t\in\Lambda_{\rho}}
   	\bint_{B_{\rho}^{(\theta)}} 
   	\frac{\big|\power{u}{\frac{1+m}{2}}(t)-
   	(\power{u}{\frac{1+m}{2}})_{\rho}^{(\theta)}\big|^2}
 	{\rho^{\frac{1+m}{m}}} \,\dx\\
    &\qquad\le
    c\,\biint_{Q_{2\rho}^{(\theta)}}
    \bigg[\frac{|u|^{1+m}}{\rho^{\frac{1+m}{m}}} +
    \frac{|u|^{2m}+\big|(\power{u}{\frac{1+m}{2}})_\rho^{(\theta)}\big|^{\frac{4m}{1+m}}}{\theta^{\frac{2m(m-1)}{1+m}}\rho^2} +
	|F|^{2}\bigg]\dx\dt\\
    &\qquad\le
    c\,\biint_{Q_{2\rho}^{(\theta)}}
    \frac{|u|^{1+m}}{\rho^{\frac{1+m}{m}}} \,\dx\dt +
    c\,\bigg[\biint_{Q_{2\rho}^{(\theta)}}
    \frac{\theta^{1-m}|u|^{1+m}}{\rho^{\frac{1+m}{m}}} \,\dx\dt\bigg]^{\frac{2m}{1+m}} +
    c\,\biint_{Q_{2\rho}^{(\theta)}} 
	|F|^{2} \,\dx\dt \\
    &\qquad\le
    c\,\theta^{2m} +
    c\,\biint_{Q_{2\rho}^{(\theta)}} 
	|F|^{2} \,\dx\dt,
\end{align*}
where $c=c(m,\nu,L)$.       
Joining the preceding inequalities leads us to 
\begin{align*}
  	\theta^m
  	\le
  	c\,\sqrt{\epsilon}\, \theta^m +
	\bigg[\biint_{Q_{\rho/2}^{(\theta)}}
    \frac{|u|^{1+m}}{\rho^{\frac{1+m}{m}}} \,\dx\dt 
    \bigg]^{\frac12} +
  	\frac{c}{\epsilon^{\frac{1}{n}}}
   	\bigg[\biint_{Q_{2\rho}^{(\theta)}} 
   	\big[|D\power{u}{m}|^{2}+|F|^2\big] \dx\dt\bigg]^{\frac12} .
\end{align*}
After choosing $\epsilon=\epsilon(n,m,L)\in(0,1]$ so small that
$c\sqrt{\epsilon}\le
  1-2^{-\frac{1}{2m}}
$, we can re-absorb the first term of
the right-hand side into the left. 
In this way, we obtain 
\begin{align*}
  	\theta^m
  	\le
          2^{\frac{1}{2m}}
        \bigg[\biint_{Q_{\rho/2}^{(\theta)}}
    \frac{|u|^{1+m}}{\rho^{\frac{1+m}{m}}} \,\dx\dt 
    \bigg]^{\frac12} +
  	c\,\bigg[\biint_{Q_{2\rho}^{(\theta)}} 
   	\big[|D\power{u}{m}|^{2}+|F|^2\big] \dx\dt\bigg] ,
\end{align*}
which proves the asserted inequality. 
\end{proof}

\section{Proof of the higher integrability}\label{sec:hi}


We consider a fixed cylinder
$$
	Q_{8R}(y_o,\tau_o)
	\equiv
	B_{8R}(y_o)\times \big(\tau_o-(8R)^\frac{1+m}{m},\tau_o+(8R)^\frac{1+m}{m}\big)
	\subseteq\Omega_T
$$
with $R>0$. Again, we omit the center in the notation and write
$Q_{\rho}:=Q_{\rho}(y_o,\tau_o)$ for short, for any radius
$\rho\in(0,8R]$. We consider a parameter
\begin{equation}\label{lambda-0}
  \lambda_o
  \ge
  1+\bigg[\biint_{Q_{4R}} 
  \frac{\abs{u}^{1+m}}{(4 R)^{\frac{1+m}m}}\d x\d t
  \bigg]^{\frac{d}{2m}},
\end{equation}
which will be fixed later. Recall that the {\it scaling deficit} $d$ is defined in \eqref{def:d}.
We again point out that the assumption $m>m_c$ ensures
$d>0$. Furthermore, for $n\ge 2$ the scaling deficit blows up
when $m\downarrow m_c$. 
For $z_o\in Q_{2R}$, $\rho\in(0,R]$, and $\theta\ge 1$, we consider
space-time cylinders $Q_\rho^{(\theta)}(z_o)$ as defined in
\eqref{def-Q}. Note that these cylinders depend monotonically on $\theta$
in the sense that $Q_\rho^{(\theta_2)}(z_o)\subset
Q_\rho^{(\theta_1)}(z_o)$ whenever $1\le \theta_1<\theta_2$,
and that $Q_\rho^{(\theta)}(z_o)\subset Q_{4R}$ for $z_o\in Q_{2R}$, $\rho\in(0,R]$, and $\theta\ge 1$.

\subsection{Construction of a non-uniform system of cylinders}\label{sec:cylinders}
The following construction of a non-uniform system of cylinders is similar to the one in \cite{Gianazza-Schwarzacher, Schwarzacher}. Let $z_o\in Q_{2R}$. For a radius $\rho\in (0,R]$ we define
$$
	\widetilde\theta_\rho
	\equiv
	\widetilde\theta_{z_o;\rho}
	:=
	\inf\bigg\{\theta\in[\lambda_o,\infty):
	\frac{1}{|Q_\rho|}
	\iint_{Q^{(\theta)}_\rho(z_o)} 
	\frac{\abs{u}^{1+m}}{\rho^{\frac{1+m}m}} \dx\dt 
	\le 
	\theta^{\frac{2m}{d}} \bigg\}.
$$
We note  that
$\widetilde\theta_\rho$ is well defined, since the
infimum in the definition is taken over a non-empty set. In fact,
in the limit $\theta\to\infty$ the integral on the left-hand side
converges to zero (and is constant in the case $m=1$, respectively), while the right-hand side grows 
with speed $\theta^{\frac{2m}{d}}$.
The choice of the exponent on the right-hand side becomes more clear after taking means in the
integral condition, since then the condition takes the form 
$$
  \biint_{Q^{(\theta)}_\rho(z_o)}
  \frac{\abs{u}^{1+m}}{\rho^{\frac{1+m}m}}\dx\dt 
  \le 
  \theta^{2m};
$$
compare Sections~\ref{sec:poin} and~\ref{sec:revholder}. As an immediate consequence of the definition of
$\widetilde\theta_\rho$, we either have 
$$
	\widetilde\theta_\rho=\lambda_o
	\quad\mbox{and}\quad
	\biint_{Q_{\rho}^{(\widetilde\theta_\rho)}(z_o)}
    \frac{\abs{u}^{1+m}}{\rho^{\frac{1+m}m}} \dx\dt
	\le
	\widetilde\theta_\rho^{2m}
	=
	\lambda_o^{2m},
$$
or
\begin{equation}\label{theta>lambda}
	\widetilde\theta_\rho>\lambda_o
	\quad\mbox{and}\quad
	\biint_{Q_{\rho}^{(\widetilde\theta_\rho)}(z_o)} 
	\frac{\abs{u}^{1+m}}{\rho^{\frac{1+m}m}}\dx\dt
	=
	\widetilde\theta_\rho^{2m}.
\end{equation}
In the case $\rho=R$,
we have $\widetilde \theta_{R}\ge \lambda_o\ge 1$.
Moreover, in the case $\widetilde\theta_{R}>\lambda_o$, property
\eqref{theta>lambda},
the inclusion $Q_{R}^{(\widetilde\theta_{R})}(z_o)\subset Q_{4R}$ and \eqref{lambda-0} yield that
\begin{align*}
	\widetilde\theta_{R}^{\frac{2m}{d}}
	&=
	\frac{1}{|Q_{R}|}
	\iint_{Q_{R}^{(\widetilde\theta_{R})}(z_o)}
	\frac{\abs{u}^{1+m}}{R^{\frac{1+m}m}} \dx\dt\\ 
	&\le
	\frac{4^{\frac{1+m}m}}{|Q_{R}|}
	\iint_{Q_{4R}} \frac{\abs{u}^{1+m}}{(4R)^{\frac{1+m}m}}\d x\d t
	\le 
	4^{n+2+\frac2m}\lambda_o^{\frac{2m}{d}},
\end{align*}
from which we infer  the bound
\begin{align}\label{bound-theta-R}
  \widetilde\theta_{R} 
  \le
  4^{\frac d{2m}(n+2+\frac2m)}\lambda_o.
\end{align}     
Our next goal is to prove the continuity of the mapping $(0,R]\ni\rho\mapsto
\widetilde\theta_\rho$.
For $\rho\in(0,R]$ and  $\eps>0$, we abbreviate 
$\theta_+:=\widetilde\theta_\rho+\eps$. We first observe
that 
there exists $\delta=\delta(\eps,\rho)>0$ such that
for all radii $r\in(0,R]$ with $|r-\rho|<\delta$ there holds
\begin{equation}\label{claim-cont}
  \frac{1}{|Q_r|}
  \iint_{Q_{r}^{(\theta_+)}(z_o)} \frac{\abs{u}^{1+m}}{r^{\frac{1+m}m}} \dx\dt 
  <
  \theta_+^{\frac{2m}{d}}.
\end{equation}
In fact, if 
$r=\rho$, this is a consequence of the definition of
$\widetilde\theta_\rho$, since $\theta_+^{\frac{2m}{d}}>\widetilde\theta_\rho^{\frac{2m}{d}}$.
By the absolute continuity of the integral, the inequality
\eqref{claim-cont} continues to hold for radii $r$ sufficiently close
to $\rho$. Hence, the definition of $\widetilde\theta_r$ implies that
$\widetilde\theta_r<\theta_+=\widetilde\theta_\rho+\eps$, provided $|r-\rho|<\delta$. 
For the corresponding lower bound
$\widetilde\theta_r>\theta_-:=\widetilde\theta_\rho-\eps$,
we proceed similarly.
First, we note that we can assume $\theta_-\ge\lambda_o$ and hence $\widetilde\theta_\rho>\lambda_o$, since
otherwise the claim immediately follows  from the property
$\widetilde\theta_r\ge\lambda_o$.
Now, we claim that 
\begin{equation}\label{claim-cont-2}
  	\frac{1}{|Q_r|}
	\iint_{Q_{r}^{(\theta_-)}(z_o)} \frac{\abs{u}^{1+m}}{r^{\frac{1+m}m}} \dx\dt 
  	>
  	\theta_-^{\frac{2m}{d}}
\end{equation}
for all $r\in(0,R]$ with $|r-\rho|<\delta$, after diminishing
$\delta=\delta(\eps,\rho)>0$ if necessary.
Again, we first consider the case $r=\rho$, in which the claim follows
from the definition of $\widetilde\theta_\rho$. In fact, if the claim
did not hold, we
would arrive at the contradiction $\widetilde\theta_\rho\le\theta_-$.
Now, for radii
$r$ with $|r-\rho|<\delta$ the assertion 
follows from the continuous dependence
of the left-hand side upon $r$.
Having established \eqref{claim-cont-2}, we can conclude from 
the definition of $\widetilde\theta_r$ that
$\widetilde\theta_r>\theta_-=\widetilde\theta_\rho-\eps$. Altogether
we have shown that $\widetilde\theta_\rho-\eps<
\widetilde\theta_r< \widetilde\theta_\rho+\eps$  for all radii $r\in(0,R]$
with $|r-\rho|<\delta$, which
completes the proof of the continuity of $(0,R]\ni\rho\mapsto
\widetilde\theta_\rho$.

Unfortunately, the mapping $(0,R]\ni\rho\to \widetilde\theta_\rho$ might not be
decreasing.
For this reason we work with a modified version of
$\widetilde\theta_\rho$, which we denote by $\theta_\rho$. This
modification is done by a rising sun type construction. More precisely, we
define 
$$
	\theta_\rho
	\equiv
	\theta_{z_o;\rho}
	:=
	\max_{r\in[\rho,R]} \widetilde\theta_{z_o;r}\,. 
$$
As an immediate consequence of the construction, the mapping
$(0,R]\ni\rho\mapsto \theta_\rho$ is
continuous and monotonically decreasing. 
In general, the modified cylinders $Q_{\rho}^{(\theta_\rho)}(z_o)$ cannot 
be expected to be intrinsic in the sense of \eqref{theta>lambda}.  
However, we can show that the cylinders $Q_{s}^{(\theta_\rho)}(z_o)$
are sub-intrinsic for all radii $s\ge\rho$. More precisely, we have 
\begin{align}\label{sub-intrinsic-2}
	\biint_{Q_{s}^{(\theta_{\rho})}(z_o)} 
	\frac{\abs{u}^{1+m}}{s^{\frac{1+m}m}} \dx\dt
	\le 
	\theta_\rho^{2m}
	\quad\mbox{for any $0<\rho\le s\le R$.}
\end{align}
For the proof of this inequality, we use the chain of inequalities 
$\widetilde\theta_s\le \theta_{s}\le \theta_{\rho}$, which implies
$Q_{s}^{(\theta_{\rho})}(z_o)\subseteq
Q_{s}^{(\widetilde\theta_{s})}(z_o)$, and the fact that the latter
cylinder is sub-intrinsic. In this way, we deduce 
\begin{align*}
	\biint_{Q_{s}^{(\theta_{\rho})}(z_o)} 
	\frac{\abs{u}^{1+m}}{s^{\frac{1+m}m}} \dx\dt
	&\le 
	\Big(\frac{\theta_{\rho}}{\widetilde\theta_{s}}\Big)^{2m-\frac{2m}{d}}
	\biint_{Q_{s}^{(\widetilde\theta_{s})}(z_o)}
    \frac{\abs{u}^{1+m}}{s^{\frac{1+m}m}} \dx\dt \\
	&\le 
	\Big(\frac{\theta_{\rho}}{\widetilde\theta_{s}}\Big)^{2m-\frac{2m}{d}}
	\widetilde\theta_{s}^{2m}
	=
	\theta_\rho^{2m-\frac{2m}{d}}\,\widetilde\theta_s^{\frac{2m}{d}}
	\le 
	\theta_\rho^{2m},
\end{align*}
which is exactly assertion~\eqref{sub-intrinsic-2}. 
Next, we define 
\begin{equation}\label{rho-tilde}
	\widetilde\rho
	:=
	\left\{
	\begin{array}{cl}
	R, &
	\quad\mbox{if $\theta_\rho=\lambda_o$,} \\[5pt]
	\min\big\{s\in[\rho, R]: \theta_s=\widetilde \theta_s \big\}, &
	\quad\mbox{if $\theta_\rho>\lambda_o$.}
	\end{array}
	\right.
\end{equation}
By definition, for any $s\in [\rho,\widetilde\rho]$ we have
$\theta_s=\widetilde\theta_{\widetilde\rho}$.
Our next goal is the proof of the upper bound 
\begin{align}\label{bound-theta}
  \theta_\rho 
  \le
  \Big(\frac{s}{\rho}\Big)^{\frac d{2m}(n+2+\frac2m)}
  \theta_{s}  
  \quad\mbox{for any $s\in(\rho,R]$.}
\end{align}
In the case $\theta_\rho=\lambda_o$ this is immediate since
$\theta_s\ge\lambda_o$. Another easy case is that of radii
$s\in(\rho,\widetilde\rho]$, since then we have
$\theta_s=\widetilde\theta_{\widetilde\rho}=\theta_\rho$. Therefore, it
only remains to prove \eqref{bound-theta}  
for the case $\theta_\rho>\lambda_o$ and radii $s\in(\widetilde\rho,R]$.
To this end, we use the monotonicity of $\rho\mapsto\theta_\rho$, \eqref{theta>lambda} and
\eqref{sub-intrinsic-2} to conclude 
\begin{align*}
	\theta_\rho^{\frac{2m}{d}}
	&=
	\widetilde \theta_{\widetilde\rho}^{\frac{2m}{d}}
	=
	\frac{1}{|Q_{\widetilde\rho}|}
	\iint_{Q_{\widetilde\rho}^{(\theta_{\widetilde\rho})}(z_o)}
	\frac{\abs{u}^{1+m}}{\widetilde\rho^{\frac{1+m}m}} \dx\dt \\
	&\le
	\Big(\frac{s}{\widetilde\rho}\Big)^{n+2+\frac2m}
	\frac{1}{|Q_{s}|} 
	\iint_{Q_{s}^{(\theta_{s})}(z_o)} 
	\frac{\abs{u}^{1+m}}{s^{\frac{1+m}m}} \dx\dt 
	\le
	\Big(\frac{s}{\rho}\Big)^{n+2+\frac2m}
	\theta_{s}^{\frac{2m}{d}} .
\end{align*}
This yields the claim~\eqref{bound-theta} also  in the remaining case.
We now apply \eqref{bound-theta} with $s=R$. Using moreover the fact
$\theta_{R}=\widetilde\theta_{R}$ and
estimate \eqref{bound-theta-R} for $\widetilde\theta_{R}$, we deduce 
\begin{align}\label{bound-theta-2}
	\theta_\rho 
	\le
	\Big(\frac{R}{\rho}\Big)^{\frac d{2m}(n+2+\frac2m)}
	\theta_{R} 
	\le
	\Big(\frac{4R}{\rho}\Big)^{\frac d{2m}(n+2+\frac2m)}
	\lambda_o
\end{align}
for every $\rho\in(0,R]$.       
In summary, for every $z_o\in Q_{2R}$, we have constructed a
system of concentric sub-intrinsic cylinders
$Q_{\rho}^{(\theta_{z_o;\rho})}(z_o)$ with radii $\rho\in (0,R]$. As a
consequence of the monotonicity of $\rho\mapsto \theta_{z_o;\rho}$, these
cylinders are nested in the sense that
$$
	\mbox{$Q_{r}^{(\theta_{z_o;r})}(z_o)
	\subset
	Q_{s}^{(\theta_{z_o;s})}(z_o)$ whenever $0<r<s\le R$.}
$$
However, keep in mind that in general these cylinders are not intrinsic but only sub-intrinsic.

\subsection{Covering property}\label{sec:covering}
Our next goal is to establish the following Vitali type covering property for the cylinders constructed in the last section. 

\begin{lemma}\label{lem:vitali}
There exists a constant $\hat c=\hat c(n,m)\ge 20$ 
such that, whenever  $\mathcal F$ is a collection of cylinders
$Q_{4r}^{(\theta_{z;r})}(z)$,
where $Q_{r}^{(\theta_{z;r})}(z)$ is a cylinder of the form constructed in Section {\upshape\ref{sec:cylinders}} with radius $r\in(0,\tfrac{R}{\hat c})$, then there exists a countable subfamily $\mathcal G$ of disjoint cylinders in $\mathcal F$ such that  
\begin{equation}\label{covering}
	\bigcup_{Q\in\mathcal F} Q
	\subseteq 
	\bigcup_{Q\in\mathcal G} \widehat Q,
\end{equation}
where $\widehat Q$ denotes the $\frac{\hat c}{4}$-times enlarged cylinder $Q$, i.e.~if $Q=Q_{4r}^{(\theta_{z;r})}(z)$, then $\widehat Q=Q_{\hat c r}^{(\theta_{z;r})}(z)$.
\end{lemma}

\begin{proof}
For $j\in\N$ we subdivide $\mathcal{F}$ into the subfamilies 
$$
  \mathcal F_j
  :=
  \big\{Q_{4r}^{(\theta_{z;r})}(z)\in \mathcal F: 
  \tfrac{R}{2^j\hat c}<r\le \tfrac{R}{2^{j-1}\hat c} \big\}.
$$
Then, we choose finite subfamilies $\mathcal G_j\subset \mathcal F_j$
according to the following scheme.
We start by choosing $\mathcal G_1$ as an arbitrary maximal disjoint collection of cylinders in
$\mathcal F_1$. The subfamily $\mathcal G_1$ is finite, since
\eqref{bound-theta-2} and the definition of $\mathcal F_1$ imply a
lower bound on the volume of each cylinder in $\mathcal G_1$.
Now, assuming that the subfamilies $\mathcal G_1, \mathcal G_2,
\dots, \mathcal G_{k-1}$ have already been constructed for some integer
$k\ge 2$, we choose $\mathcal G_k$ to be any maximal disjoint subcollection of 
$$
	\Bigg\{Q\in \mathcal F_k: 
	Q\cap Q^\ast=\emptyset \mbox{ for any $ \displaystyle Q^\ast\in \bigcup_{j=1}^{k-1} \mathcal G_j $}
	\Bigg\}.
$$
For the same reason as above, the collection $\mathcal G_k$ is
finite. Hence, the family 
$$
	\mathcal G
	:=
	\bigcup_{j=1}^\infty \mathcal G_j\subseteq\mathcal{F}
$$
defines a countable collection of disjoint cylinders.
It remains to prove that
for each cylinder  $Q\in\mathcal F$ there exists a cylinder $Q^\ast\in\mathcal
G$ with  $Q\subset \widehat {Q}^\ast$.  
To this end, we fix a cylinder
$Q=Q_{4r}^{(\theta_{z;r})}(z)\in\mathcal F$.
Let $j\in\N$ be such that $Q\in\mathcal F_j$.
The maximality of $\mathcal G_j$ ensures the existence of a cylinder
$Q^\ast=Q_{4r_\ast}^{(\theta_{z_\ast;r_\ast})}(z_\ast)\in
\bigcup_{i=1}^{j} \mathcal G_i$ with $Q\cap Q^\ast\not=\emptyset$.
We will show that this cylinder has the desired property $Q\subset \widehat {Q}^\ast$.
First, we observe that the properties $r\le\tfrac{R}{2^{j-1}\hat c}$
and $r_\ast>\tfrac{R}{2^j\hat c}$ imply $r\le 2r_\ast$, which ensures
$\Lambda_{4r}(t)\subseteq \Lambda_{20r_\ast}(t_\ast)$. For the proof of the corresponding spatial inclusion
$B^{(\theta_{z,r})}_{4r}(x)\subseteq B_{\hat c r_\ast}^{(\theta_{z_\ast,r_\ast})}(x_\ast)$, we
first shall derive the bound 
\begin{equation}\label{control-theta-1}
  \theta_{z_\ast;r_\ast}
  \le
  52^{\frac d{2m}(n+2+\frac2m)}\,
  \theta_{z;r}\,.
\end{equation}
We recall the definition \eqref{rho-tilde} of the radius $\widetilde
r_\ast\in [r_\ast,R]$ which is 
associated to the cylinder
$Q_{r_\ast}^{(\theta_{z_\ast;r_\ast})}(z_\ast)$. According to the
definition, we either have that $Q_{\widetilde
  r_\ast}^{(\theta_{z_\ast;r_\ast})}(z_\ast)$ is intrinsic
or  that $\widetilde r_\ast=R$ and $\theta_{z_\ast;r_\ast}=\lambda_o$. In
the second alternative, the claim~\eqref{control-theta-1} is
immediate, since 
$$
	\theta_{z_\ast;r_\ast}
	=
	\lambda_o
	\le 
	\theta_{z;r}\,.
$$
Therefore, it remains to consider the case that $Q_{\widetilde
  r_\ast}^{(\theta_{z_\ast;r_\ast})}(z_\ast)$ is intrinsic in the
sense that 
\begin{align}\label{control-theta-2}
	\theta_{z_\ast;r_\ast}^{\frac{2m}{d}}
	=
	\frac{1}{|Q_{\widetilde r_\ast}|}
	\iint_{Q_{\widetilde r_\ast}^{(\theta_{z_\ast;r_\ast})}(z_\ast)}
	\frac{\abs{u}^{1+m}}{\widetilde r_\ast^{\frac{1+m}m}} \dx\dt.
\end{align}
We distinguish between the cases  $\widetilde r_\ast\le \frac{R}{\mu}$ and $\widetilde r_\ast> \frac{R}{\mu}$, where $\mu:= 13$. 
We start with the latter case. Using \eqref{control-theta-2} and the
definition of $\lambda_o$ and $\theta_{z;r}$, we estimate
\begin{align*}
	\theta_{z_\ast;r_\ast}^{\frac{2m}{d}}
	&\le
	\Big(\frac{4R}{\widetilde r_\ast}\Big)^{\frac{1+m}m} \frac{1}{|Q_{\widetilde r_\ast}|}
	\iint_{Q_{4R}}
	\frac{\abs{u}^{1+m}}{(4R)^{\frac{1+m}m}} \dx\dt \\
	&\le
	\Big(\frac{4 R}{\widetilde r_\ast}\Big)^{n+2+\frac2m} 
	\lambda_o^{\frac{2m}{d}}\\
	&\le
	(4\mu)^{n+2+\frac2m} \theta_{z;r}^{\frac{2m}{d}},
\end{align*}
which can be rewritten in the form
\begin{align*}
	\theta_{z_\ast;r_\ast}
	\le
	(4\mu)^{\frac d{2m}(n+2+\frac2m)}\,\theta_{z;r}\,.
\end{align*}
This yields \eqref{control-theta-2} in the second case, and it only
remains to consider the first case $\widetilde r_\ast\le \frac{R}{\mu}$. Here, the
key step is to prove the inclusion
\begin{equation}\label{incl-cyl}
  Q_{\widetilde r_\ast}^{(\theta_{z_\ast;r_\ast})}(z_\ast)
  \subseteq 
  Q_{\mu\widetilde r_\ast}^{(\theta_{z;\mu\widetilde r_\ast})}(z).
\end{equation}
We first observe
that $\widetilde r_\ast\ge r_\ast$ and
$|t-t_\ast|<(4r)^{\frac{1+m}m}+(4r_\ast)^{\frac{1+m}m}\le
(12r_\ast)^{\frac{1+m}m}$ implies $\Lambda_{\widetilde r_\ast}(t_\ast)\subseteq \Lambda_{\mu\widetilde
  r_\ast}(t)$.  In addition, we have 
\begin{equation}\label{x-x_ast}
	|x-x_\ast|\le \theta_{z;r}^{\frac{m(m-1)}{1+m}}4r
	+
	\theta_{z_\ast;r_\ast}^{\frac{m(m-1)}{1+m}}4r_\ast.
\end{equation}
At this point, we may assume that
$\theta_{z;r}\le\theta_{z_\ast;r_\ast}$, since \eqref{control-theta-1}
clearly is satisfied in the alternative  case.
Then, the monotonicity of $\rho\mapsto \theta_{z;\rho}$  and 
$r\le 2r_\ast\le2\widetilde r_\ast\le\mu\widetilde r_\ast$ imply
$$
	\theta_{z_\ast;r_\ast}
	\ge 
	\theta_{z;r}
	\ge
	\theta_{z;\mu \widetilde r_\ast}.
$$
Combining this with \eqref{x-x_ast}, we conclude  that
\begin{align*}
	\theta_{z_\ast;r_\ast}^{\frac{m(m-1)}{1+m}}\widetilde r_\ast + |x-x_\ast|
	\le
    \theta_{z_\ast;r_\ast}^{\frac{m(m-1)}{1+m}}5\widetilde r_\ast+
    \theta_{z;r}^{\frac{m(m-1)}{1+m}}4r 
	\le
	\theta_{z;\mu\widetilde r_\ast}^{\frac{m(m-1)}{1+m}}\mu\widetilde r_\ast,
\end{align*}
from which we deduce the inclusion 
$$
  B_{\widetilde r_\ast}^{(\theta_{z_\ast;r_\ast})}(x_\ast)
  \subseteq 
  B_{\mu\widetilde r_\ast}^{(\theta_{z;\mu\widetilde r_\ast})}(x).
$$
This completes the proof of \eqref{incl-cyl}.

Using \eqref{control-theta-2}, \eqref{incl-cyl}, and
\eqref{sub-intrinsic-2} with $\rho=s=\mu\tilde r_\ast$, we estimate 
\begin{align*}
	\theta_{z_\ast;r_\ast}^{\frac{2m}{d}}
	\le
	\frac{\mu^{\frac{1+m}{m}}}{|Q_{\widetilde r_\ast}|} 
	\iint_{Q_{\mu\widetilde r_\ast}^{(\theta_{z;\mu\widetilde r_\ast})}(z)}
	\frac{\abs{u}^{1+m}}{(\mu\tilde r_\ast)^{\frac{1+m}{m}}} \dx\dt
	\le
	\mu^{n+2+\frac2m} \theta_{z;r}^{\frac{2m}{d}},
\end{align*}
which implies 
\begin{align*}
	\theta_{z_\ast;r_\ast}
	\le
	\mu^{\frac d{2m}(n+2+\frac2m)}\,\theta_{z;r}.
\end{align*}
This yields \eqref{control-theta-1} also in the last 
case.

Having established \eqref{control-theta-1}, it remains to prove the inclusion $Q_{4r}^{(\theta_{z;r})}(z)\subseteq
Q_{\hat c r_\ast}^{(\theta_{z_\ast;r_\ast})}(z_\ast)$, which will
complete the proof of the Vitali covering property. First, we note
that for any choice of $\hat c$ with $\hat c\ge 20$, we have
$\Lambda_{4r}(t)\subseteq\Lambda_{\hat c
r_\ast}(t_\ast)$. Moreover, from the facts \eqref{x-x_ast}, $r\le 2r_\ast$, and \eqref{control-theta-1} we conclude
\begin{align*}
	\theta_{z;r}^{\frac{m(m-1)}{1+m}}4r + |x-x_\ast|
	&\le
	2\theta_{z;r}^{\frac{m(m-1)}{1+m}}4r + 
	\theta_{z_\ast;r_\ast}^{\frac{m(m-1)}{1+m}}4r_\ast \\
	&\le
	4\Big[4\cdot 52^{\frac{d(1-m)[m(n+2)+2]}{2m(1+m)}}+1\Big]
	\theta_{z_\ast;r_\ast}^{\frac{m(m-1)}{1+m}}r_\ast \\
	&\le
	\theta_{z_\ast;r_\ast}^{\frac{m(m-1)}{1+m}}\hat c r_\ast,
\end{align*}
for a suitable choice of the constant $\hat c=\hat c(n,m)\ge 20$. This implies the spatial inclusion
$B_{4r}^{(\theta_{z;r})}(x)\subseteq B_{\hat c
r_\ast}^{(\theta_{z_\ast;r_\ast})}(x_\ast)$, which is the remaining
piece of information to conclude that
$$
	Q = Q_{4r}^{(\theta_{z;r})}(z)
	\subseteq
	Q_{\hat c r_\ast}^{(\theta_{z_\ast;r_\ast})}(z_\ast)
    =
	\widehat Q^\ast.
$$
Thereby we have established the inclusion \eqref{covering}, which yields
the desired Vitali type covering property.
\end{proof}

\subsection{Stopping time argument}
Now, we fix the parameter $\lambda_o$ by letting
\begin{equation*}
 	\lambda_o
 	:=
 	1+\Bigg[\biint_{Q_{4R}} 
 	\bigg[\frac{\abs{u}^{1+m}}{(4 R)^{\frac{1+m}m}} +
    |D\power{u}{m}|^2 +
 	|F|^{2}\bigg]\,\dx\dt
    \Bigg]^{\frac{d}{2m}}.
\end{equation*}
For $\lambda>\lambda_o$ and $r\in(0,2R]$, we define the super-level set
of the function $|D\power{u}{m}|$ by 
$$
	\boldsymbol E(r,\lambda)
	:=
	\Big\{z\in Q_{r}: 
	\mbox{$z$ is a Lebesgue point of $|D\power{u}{m}|$ and 
	$|D\power{u}{m}|(z) > \lambda^{m}$}\Big\}.
$$
In the definition of $\boldsymbol E(r,\lambda)$, the notion of Lebesgue points is to be understood
with regard to the system of cylinders constructed in
Section~\ref{sec:cylinders}. We point out that also with respect to
these cylinders, $\mathcal L^{n+1}$-a.e. point is a Lebesgue
point. This follows from \cite[2.9.1]{Federer}, since we already have
verified the Vitali type covering property in Lemma~\ref{lem:vitali}. 
Now, we fix radii $R\le R_1<R_2\le 2R$. Note that for any
$z_o\in Q_{R_1}$, $\kappa\ge1$ and $\rho\in(0,R_2-R_1]$ we have
\begin{equation*}
  	Q_\rho^{(\kappa)}(z_o)
	\subseteq 
	Q_{R_2}
	\subseteq 
	Q_{2R}.
\end{equation*}
For the following argument, we restrict ourselves to levels $\lambda$
with 
\begin{equation}\label{choice_lambda}
	\lambda
	>
	B\lambda_o,
	\quad\mbox{where}
	\quad
	B
	:=
	\Big(\frac{4\hat c R}{R_2-R_1}\Big)^{\frac{d(n+2)(1+m)}{(2m)^2}}
	>1,
\end{equation}
and where $\hat c=\hat c(n,m)$ is the constant from the Vitali-type
covering Lemma \ref{lem:vitali}.
We fix $z_o\in \boldsymbol E(R_1,\lambda)$ and abbreviate
$\theta_s\equiv \theta_{z_o;s}$ for $s\in(0,R]$ throughout this
section. By definition of $\boldsymbol E(R_1,\lambda)$, we have
\begin{equation}\label{larger-lambda}
	\liminf_{s\downarrow 0} 
	\biint_{Q_{s}^{(\theta_{s})}(z_o)} 
	\big[|D\power{u}{m}|^2 + |F|^2 \big] 
	\dx\dt
	\ge
	|D\power{u}{m}|^2(z_o)
	>
	\lambda^{2m}.
\end{equation}
On the other hand, for any radius $s$ with
\begin{align}\label{radius-s}
	\frac{R_2-R_1}{\hat c}\le s\le R
\end{align}
the definition of $\lambda_o$, estimate
\eqref{bound-theta-2}, assumption \eqref{radius-s} and the definition
of $d$ imply 
\begin{align}\label{smaller-lambda}
	\biint_{Q_{s}^{(\theta_s)}(z_o)} &
	\big[|D\power{u}{m}|^2+|F|^2\big] 
	\dx\dt \nonumber\\
	&\le
	\frac{|Q_{4R}|}{|Q_{s}^{(\theta_s)}|}
	\biint_{Q_{4R}} \big[|D\power{u}{m}|^2 + |F|^2\big] \dx\dt \nonumber\\
	&\le
	\frac{|Q_{4R}|}{|Q_{s}|}\,
	\theta_s^{\frac{nm(1-m)}{1+m}}
	\lambda_o^{\frac{2m}{d}} \nonumber\\
	&\le
	\Big(\frac{4R}{s}\Big)^{n+1+\frac{1}m+\frac{d}{2m}(n+2+\frac2m)\cdot\frac{nm(1-m)}{1+m}} 
	\lambda_o^{\frac{nm(1-m)}{1+m}+\frac{2m}d} \nonumber\\
        &=
	\Big(\frac{4R}{s}\Big)^{\frac{d(n+2)(1+m)}{2m}} 
	\lambda_o^{2m} \nonumber\\
	&\le
	B^{2m} \lambda_o^{2m} 
	<
	\lambda^{2m}.
\end{align}
By the continuity of the mapping $s\mapsto\theta_s$ and the absolute
continuity of the integral, the left-hand side of
\eqref{smaller-lambda} depends continuously on $s$. Therefore, in view
of \eqref{larger-lambda} and \eqref{smaller-lambda}, there exists
a maximal radius $0<\rho_{z_o} < \tfrac{R_2-R_1}{\hat c}$ for which
the above inequality becomes an equality, i.e.~$\rho_{z_o}$ is the
maximal radius with 
\begin{align}\label{=lambda}
	\biint_{Q_{\rho_{z_o}}^{(\theta_{\rho_{z_o}})}(z_o)} 
	\big[|D\power{u}{m}|^2 + |F|^2\big] \dx\dt
	=
	\lambda^{2m}.
\end{align}
The maximality of the radius $\rho_{z_o}$ implies in particular that 
\begin{align*}
	\biint_{Q_{s}^{(\theta_{s})}(z_o)} 
	\big[|D\power{u}{m}|^2 + |F|^2\big] \dx\dt
	<
	\lambda^{2m}
	\quad
	\mbox{for any $s\in (\rho_{z_o}, R]$.}
\end{align*}
Due to the monotonicity of the mapping $\rho\mapsto \theta_\rho$ and \eqref{bound-theta} we have 
\begin{align*}
	\theta_\sigma
	\le 
	\theta_{s}
	\le
	\Big(\frac{\sigma}{s}\Big)^{\frac{d}{2m}(n+2+\frac2m)}
  	\theta_{\sigma}
	\quad
	\mbox{for any $\rho_{z_o}\le s<\sigma\le R$,}
\end{align*}
so that 
\begin{align}\label{<lambda}
	\biint_{Q_{\sigma}^{(\theta_{s})}(z_o)} 
	\big[|D\power{u}{m}|^2 + |F|^2\big] \dx\dt
	&\le
	\Big(\frac{\theta_{s}}{\theta_\sigma}\Big)^{\frac{nm(1-m)}{1+m}}
	\biint_{Q_{\sigma}^{(\theta_{\sigma})}(z_o)} 
	\big[|D\power{u}{m}|^2 + |F|^2\big] \dx\dt \nonumber\\
	&<
	\Big(\frac{\sigma}{s}\Big)^{\frac{dn(1-m)}{2(1+m)}(n+2+\frac2m)}\,
	\lambda^{2m}
\end{align}
for any $\rho_{z_o}\le s<\sigma\le R$. 
Finally, we recall that the cylinders are constructed in such a
way that 
\begin{equation*}
  Q_{\hat c\rho_{z_o}}^{(\theta_{\rho_{z_o}})}(z_o)
  \subseteq
  Q_{\hat c\rho_{z_o}}(z_o)
  \subseteq
  Q_{R_2}.
\end{equation*}

\subsection{A Reverse H\"older Inequality}
For a level  $\lambda$ as in \eqref{choice_lambda} and a point $z_o\in
\boldsymbol E(R_1,\lambda)$, we consider the radius
$\widetilde\rho_{z_o}\in[\rho_{z_o},R]$ as defined in
\eqref{rho-tilde}. In the sequel
we write $\theta_{\rho_{z_o}}$ instead of $\theta_{z_o;\rho_{z_o}}$.
We recall that $\widetilde\rho_{z_o}$ has been defined in such a way
that for any $s\in [\rho_{z_o},
\widetilde\rho_{z_o}]$ we have
$\theta_s=\theta_{\rho_{z_o}}$,  and, in particular,
$\theta_{\widetilde\rho_{z_o}}=\theta_{\rho_{z_o}}$.

The aim of this section is the proof of a reverse
H\"older inequality on 
$Q_{2\rho_{z_o}}^{(\theta_{\rho_{z_o}})}(z_o)$. To this end, we need to
verify the assumptions of Proposition \ref{prop:revhoelder}. 
First, we note that \eqref{sub-intrinsic-2} with
$s=4\rho_{z_o}$ implies
\begin{equation*}
  \biint_{Q_{4\rho_{z_o}}^{(\theta_{\rho_{z_o}})}(z_o)} 
  \frac{\abs{u}^{1+m}}{(4\rho_{z_o})^{\frac{1+m}m}} \dx\dt
  \le 
  \theta_{\rho_{z_o}}^{2m},
\end{equation*}
which means that assumption \eqref{sub-intrinsic} is fulfilled for the cylinder
$Q_{2\rho_{z_o}}^{(\theta_{\rho_{z_o}})}(z_o)$ with $K=1$.
For the estimate of $\theta_{\rho_{z_o}}^{2m}$ from above,  
we distinguish between the cases $\widetilde\rho_{z_o}\le 2\rho_{z_o}$
and $\widetilde\rho_{z_o}> 2\rho_{z_o}$.
In the former case, we use the fact
$\theta_{\rho_{z_o}}=\theta_{\widetilde\rho_{z_o}}=\widetilde\theta_{\widetilde\rho_{z_o}}$, which implies
that 
$Q_{\widetilde\rho_{z_o}}^{(\theta_{\rho_{z_o}})}(z_o)$ is intrinsic, 
and then the bound $\widetilde\rho_{z_o}\le 2\rho_{z_o}$, with the
result that
\begin{equation*}
  \theta_{\rho_{z_o}}^{2m}
  =
  \biint_{Q_{\widetilde\rho_{z_o}}^{(\theta_{\rho_{z_o}})}(z_o)}
  \frac{\abs{u}^{1+m}}{\widetilde\rho_{z_o}^{\frac{1+m}m}}
  \dx\dt
  \le
  2^{n+2+\frac2m}
  \biint_{Q_{2\rho_{z_o}}^{(\theta_{\rho_{z_o}})}(z_o)}
  \frac{\abs{u}^{1+m}}{(2\rho_{z_o})^{\frac{1+m}m}}
  \dx\dt.
\end{equation*}
This means that in this case, assumption \eqref{super-intrinsic}$_1$
is satisfied with $K\equiv  2^{n+2+\frac2m}$.
Next, we consider the remaining case 
$\widetilde\rho_{z_o}>2\rho_{z_o}$.
Here, we claim that
\begin{equation}
  \label{claim-case-2}
  \theta_{\rho_{z_o}}\le c(n,m,\nu,L)\lambda.
\end{equation}
For the proof we treat the cases 
$\widetilde\rho_{z_o}\in(2\rho_{z_o},\frac R2]$ and
$\widetilde\rho_{z_o}\in(\frac R2,R]$ separately. 
In the latter case, we use
\eqref{bound-theta-2} with $\rho=\widetilde\rho_{z_o}$ and  
the bound $\widetilde\rho_{z_o}>\frac R2$ in order to estimate
$\theta_{\rho_{z_o}}=\theta_{\widetilde\rho_{z_o}}\le c\lambda_o \le c\lambda$, 
which yields \eqref{claim-case-2}. 
In the alternative case $\widetilde\rho_{z_o}\in(2\rho_{z_o},\frac
R2]$, the cylinder
$Q_{\widetilde\rho_{z_o}}^{(\theta_{\rho_{z_o}})}(z_o)$ is intrinsic
by definition of $\widetilde\rho_{z_o}$, and the two times enlarged  cylinder
$Q_{2\widetilde\rho_{z_o}}^{(\theta_{\rho_{z_o}})}(z_o)$ is
sub-intrinsic by \eqref{sub-intrinsic-2}.
Therefore, assumptions \eqref{sub-intrinsic} and \eqref{super-intrinsic}$_1$ of Lemma~\ref{lem:theta} are satisfied for $Q_{\widetilde\rho_{z_o}}^{(\theta_{\rho_{z_o}})}(z_o)$. The application of the lemma yields
\begin{equation*}
  	\theta_{\rho_{z_o}}^m
  	\le
	\tfrac1{\sqrt2} \Bigg[
	\biint_{Q_{\widetilde\rho_{z_o}/2}^{(\theta_{\rho_{z_o}})}(z_o)}\!
    \frac{|u|^{1+m}}{(\widetilde\rho_{z_o}/2)^{\frac{1+m}{m}}} \dx\dt 
    \Bigg]^{\frac12} +
  	c \bigg[\biint_{Q_{2\widetilde\rho_{z_o}}^{(\theta_{\rho_{z_o}})}(z_o)} \!
   	\big[|D\power{u}{m}|^{2}+|F|^2\big] \dx\dt\bigg]^{\frac12}
\end{equation*}
with  a constant $c=c(n,m,L)$. For the first term, we exploit the
sub-intrinsic coupling \eqref{sub-intrinsic-2} with radii
$\rho=\rho_{z_o}$ and $s=\frac12\widetilde\rho_{z_o}>\rho_{z_o}$.
For the estimate of the last integral, we
recall that $\theta_{\widetilde\rho_{z_o}}=\theta_{\rho_{z_o}}$,
which allows to use \eqref{<lambda} with $s=\widetilde\rho_{z_o}$ and
$\sigma=2\widetilde\rho_{z_o}\in(s,R]$. This leads to the upper bound
\begin{align*}
  	\theta_{\rho_{z_o}}^m
  	\le
	\tfrac1{\sqrt2} \theta_{\rho_{z_o}}^m +
  	c\,\lambda^m .
\end{align*}
Here, we re-absorb $\tfrac1{\sqrt2} \theta_{\rho_{z_o}}^m$ into
the left and obtain the claim
\eqref{claim-case-2} in any case.

Combining this with the identity~\eqref{=lambda}, we obtain the bound  
\begin{align*}
  	\theta_{\rho_{z_o}}^{2m}
  	\le
  	c\,\lambda^{2m}
  	\le
  	c\,\biint_{Q_{2\rho_{z_o}}^{(\theta_{\rho_{z_o}})}(z_o)} 
  	\big[|D\power{u}{m}|^2 + |F|^2\big] \dx\dt.
\end{align*}
This means that in the case $\widetilde\rho_{z_o}>2\rho_{z_o}$
assumption \eqref{super-intrinsic}$_2$ is satisfied  on
$Q_{2\rho_{z_o}}^{(\theta_{\rho_{z_o}})}$ with a constant $K\equiv K(n,m,\nu,L)$.
In conclusion, in any case we have shown that all hypotheses of Proposition~\ref{prop:revhoelder} 
are satisfied. Consequently, the proposition yields the desired reverse
H\"older inequality 
\begin{align}\label{rev-hoelder}
	\biint_{Q_{2\rho_{z_o}}^{(\theta_{\rho_{z_o}})}(z_o)} & 
	|D\power{u}{m}|^2 \dx\dt \nonumber\\
	&\le
	c\bigg[\biint_{Q_{4\rho_{z_o}}^{(\theta_{\rho_{z_o}})}(z_o)} 
	|D\power{u}{m}|^{2q} \dx\dt \bigg]^{\frac{1}{q}} +
	c\, \biint_{Q_{4\rho_{z_o}}^{(\theta_{\rho_{z_o}})}(z_o)} |F|^2 \dx\dt,
\end{align}
for a constant $c=c(n,m,\nu,L)$ and with  exponent $q=\max\{\frac{n(1+m)}{2(nm+1+m)},\frac12\}<1$.

\subsection{Estimate on super-level sets}
So far we have shown that 
for every $\lambda$ as in  \eqref{choice_lambda} and every $z_o\in
\boldsymbol E(R_1,\lambda)$, there exists a  cylinder
$Q_{\rho_{z_o}}^{(\theta_{z_o;\rho_{z_o}})}(z_o)$ with 
$Q_{\hat c\rho_{z_o}}^{(\theta_{z_o;\rho_{z_o}})}(z_o)\subseteq Q_{R_2}$, 
for
which the properties \eqref{=lambda} and \eqref{<lambda},  and
the reverse H\"older type estimate \eqref{rev-hoelder} are satisfied. 
This allows us to establish
a reverse H\"older inequality for the distribution
function of $|D\power{u}{m}|^2$
by a Vitali covering type argument. The precise argument is as follows.
We define the super-level set of the inhomogeneity $F$ by 
$$
	\boldsymbol F(r,\lambda)
	:=
	\Big\{z\in Q_{r}: 
	\mbox{$z$ is a Lebesgue point of $F$ and 
	$|F|>\lambda^{m}$}\Big\},
$$
where again, the Lebesgue points have to be understood with respect to the cylinders constructed in  Section~\ref{sec:cylinders}.
Using \eqref{=lambda} and \eqref{rev-hoelder}, we estimate
\begin{align}\label{level-set-1}
	\lambda^{2m}
	&=
	\biint_{Q_{\rho_{z_o}}^{(\theta_{\rho_{z_o}})}(z_o)} 
	\big[|D\power{u}{m}|^2 + |F|^{2}\big] \d x\d t \nonumber\\
	&\le
	c\,\Bigg[\biint_{Q_{4\rho_{z_o}}^{(\theta_{\rho_{z_o}})}(z_o)} 
	|D\power{u}{m}|^{2q} \dx\dt \Bigg]^{\frac{1}{q}} +
	c\, \biint_{Q_{4\rho_{z_o}}^{(\theta_{\rho_{z_o}})}(z_o)} |F|^{2} \dx\dt 
	\nonumber\\
	&\le
	c\,\eta^{2m}\lambda^{2m} +
	c\,\Bigg[\frac{1}{\big|Q_{4\rho_{z_o}}^{(\theta_{\rho_{z_o}})}(z_o)\big|}
	\iint_{Q_{4\rho_{z_o}}^{(\theta_{\rho_{z_o}})}(z_o)\cap \boldsymbol E(R_2,\eta\lambda)} 
	|D\power{u}{m}|^{2q} \dx\dt \Bigg]^{\frac{1}{q}} \nonumber\\
	&\quad+
	\frac{c}{\big|Q_{4\rho_{z_o}}^{(\theta_{\rho_{z_o}})}(z_o)\big|}
	\iint_{Q_{4\rho_{z_o}}^{(\theta_{\rho_{z_o}})}(z_o)\cap \boldsymbol F(R_2,\eta\lambda)} 
	|F|^{2} \dx\dt ,
\end{align}
for a constant  $c=c(n,m,\nu ,L)$ and any $\eta\in(0,1)$.
We choose the parameter $\eta\in(0,1)$ in dependence on $n,m,\nu,$ and $L$
in such a way that $\eta^{2m}=\frac{1}{2c}$. This allows us
to re-absorb the  term $\frac12\lambda^{2m}$ into the
left-hand side. For the estimate of the second last term, we apply 
H\"older's inequality and \eqref{<lambda}, with the result
\begin{align*}
	\Bigg[\frac{1}{\big|Q_{4\rho_{z_o}}^{(\theta_{\rho_{z_o}})}(z_o)\big|} &
	\iint_{Q_{4\rho_{z_o}}^{(\theta_{\rho_{z_o}})}(z_o)\cap \boldsymbol E(R_2,\eta\lambda)} 
	|D\power{u}{m}|^{2q} \dx\dt \Bigg]^{\frac{1}{q}-1} \\
	&\le
	\bigg[\biint_{Q_{4\rho_{z_o}}^{(\theta_{\rho_{z_o}})}(z_o)} 
	|D\power{u}{m}|^{2} \dx\dt \bigg]^{1-q}
	\le
	c\,\lambda^{2m(1-q)}.
\end{align*}
We use this to estimate the right-hand side of \eqref{level-set-1}. The  resulting inequality is then
multiplied  by
$\big|Q_{4\rho_{z_o}}^{(\theta_{\rho_{z_o}})}(z_o)\big|$. In this way,
we arrive at
\begin{align*}
	\lambda^{2m}\big|Q_{4\rho_{z_o}}^{(\theta_{\rho_{z_o}})}(z_o)\big|
	&\le
	c\iint_{Q_{4\rho_{z_o}}^{(\theta_{\rho_{z_o}})}(z_o)\cap \boldsymbol E(R_2,\eta\lambda)} 
	\lambda^{2m(1-q)}|D\power{u}{m}|^{2q} \dx\dt \\
	&\quad+
	c
	\iint_{Q_{4\rho_{z_o}}^{(\theta_{\rho_{z_o}})}(z_o)\cap \boldsymbol F(R_2,\eta\lambda)} 
	 |F|^{2} \dx\dt,
\end{align*}
where $c=c(n,m,\nu,L)$. 
On the other hand, we bound the left-hand side
from below by the use of \eqref{<lambda}. This leads to the
inequality
\begin{align*}
	\lambda^{2m}
    &\ge
    c\,
	\biint_{Q_{\hat c\rho_{z_o}}^{(\theta_{\rho_{z_o}})}(z_o)} 
	|D\power{u}{m}|^2 \dx\dt,
\end{align*}
where we relied on the fact that $\hat c$ is a universal constant depending only on $n$ and $m$.       
Combining the two preceding estimates and using  again that $\hat
c=\hat c(n,m)$, we arrive at
\begin{align}\label{level-est}
	\iint_{Q_{\hat c\rho_{z_o}}^{(\theta_{\rho_{z_o}})}(z_o)} 
	|D\power{u}{m}|^2 \d x\d t 
	&\le
    c
    \iint_{Q_{4\rho_{z_o}}^{(\theta_{\rho_{z_o}})}(z_o)\cap \boldsymbol E(R_2,\eta\lambda)} 
	\lambda^{2m(1-q)}|D\power{u}{m}|^{2q} \dx\dt \nonumber\\
	&\quad +
	c
	\iint_{Q_{4\rho_{z_o}}^{(\theta_{\rho_{z_o}})}(z_o)\cap \boldsymbol F(R_2,\eta\lambda)} 
	|F|^{2} \dx\dt
\end{align}
for a constant $c=c(n,m,\nu ,L)$.
Since the preceding inequality holds for every center $z_o\in \boldsymbol
E(R_1,\lambda)$, we conclude that it is possible to cover the
super-level set
$\boldsymbol E(R_1,\lambda)$ by a  family
$\mathcal F\equiv\big\{Q_{4\rho_{z_o}}^{(\theta_{z_o;\rho_{z_o}})}(z_o)\big\}$ of parabolic cylinders 
with center $z_o\in \boldsymbol E(R_1,\lambda)$, such that each of the
cylinders is contained in $Q_{R_2}$, and that
on each cylinder estimate \eqref{level-est} is valid.
An application of the Vitali type Covering Lemma \ref{lem:vitali}
provides us with a countable disjoint subfamily
$$
	\Big\{Q_{4\rho_{z_i}}^{(\theta_{z_i;\rho_{z_i}})}(z_i)\Big\}_{i\in\N}
	\subseteq 
	\mathcal F
$$
with the property
$$
	\boldsymbol E(R_1,\lambda)
	\subseteq 
	\bigcup_{i=1}^\infty Q_{\hat c\rho_{z_i}}^{(\theta_{z_i;\rho_{z_i}})}(z_i)
	\subseteq
	Q_{R_2}.
$$
We apply \eqref{level-est} for each of the cylinders
$Q_{4\rho_{z_i}}^{(\theta_{z_i;\rho_{z_i}})}(z_i)$ and add the
resulting inequalities. Since the cylinders
$Q_{4\rho_{z_i}}^{(\theta_{z_i;\rho_{z_i}})}(z_i)$ are pairwise
disjoint, we obtain 
\begin{align*}
	\iint_{\boldsymbol E(R_1,\lambda)} 
	|D\power{u}{m}|^2 \dx\dt
	&\le
	c\iint_{\boldsymbol E(R_2,\eta\lambda)} 
	\lambda^{2m(1-q)}|D\power{u}{m}|^{2q} \dx\dt +
	c\iint_{\boldsymbol F(R_2,\eta\lambda)} 
	|F|^{2} \dx\dt,
\end{align*}
with $c=c(n,m,\nu,L)$. In order to compensate for the fact that
super-level sets of different levels appear on both sides of the
preceeding estimate, we need an
estimate on the difference $\boldsymbol E(R_1,\eta\lambda)\setminus
\boldsymbol E(R_1,\lambda)$.
However, on this set we can simply estimate
$|D\power{u}{m}|^2\le\lambda^{2m}$, which leads to the bound 
\begin{align*}
	\iint_{\boldsymbol E(R_1,\eta\lambda)\setminus \boldsymbol E(R_1,\lambda)} 
	|D\power{u}{m}|^2 \d x\d t
	&\le
	\iint_{\boldsymbol E(R_2,\eta\lambda)} 
	\lambda^{2m(1-q)}|D\power{u}{m}|^{2q} \dx\dt.
\end{align*}
Adding the last two inequalities, we obtain a reverse
H\"older type inequality for the distribution function of $|D\power{u}{m}|^2$ for levels $\eta\lambda$. 
In this inequality we replace  $\eta\lambda$ by $\lambda$ and recall that $\eta\in(0,1)$
was chosen as a universal constant depending only on $n,m,\nu$, and
$L$. In this way, we obtain for any $\lambda\ge \eta B\lambda_o
=:\lambda_1$ that
\begin{align}\label{pre-1}
	\iint_{\boldsymbol E(R_1,\lambda)} &
	|D\power{u}{m}|^2 \d x\d t \nonumber\\
	&\le
	c\iint_{\boldsymbol E(R_2,\lambda)} 
	\lambda^{2m(1-q)}|D\power{u}{m}|^{2q} \dx\dt +
	c \iint_{\boldsymbol F(R_2,\lambda)} |F|^{2} \dx\dt
\end{align}
with a constant $c=c(n,m,\nu ,L)$. This is the desired
reverse H\"older type inequality for the distribution function of $D\power{u}{m}$.

\subsection{Proof of the gradient estimate}
Once \eqref{pre-1} is established the final higher integrability result follows 
by integrating  \eqref{pre-1} over the range of possible values $\lambda$. Here, one has to take into account 
that the existence  of  certain integrals appearing in the proof are not guaranteed in advance. This technical point can be
overcome by use of truncation methods. Since all arguments have been elaborated in detail for example
in \cite{BDKS-higher-int, BDKS-doubly} we omit the details, and  only state the final outcome. There exists $\epsilon_o=\epsilon_o(n,m,\nu,L)\in (0,1]$ such that for any $0<\epsilon<\epsilon_1:=\min\{\epsilon_o, \sigma-2\}$ we have
\begin{align*}
	\biint_{Q_{R}} 
	|D\power{u}{m}|^{2+\epsilon} \d x\d t 
	\le
	c\, 
	\lambda_o^{\epsilon m} 
	\biint_{Q_{2R}}
	|D\power{u}{m}|^{2} \dx\dt +
	c\,\biint_{Q_{2R}}  |F|^{2+\epsilon}
	\dx\dt .
\end{align*}
At this stage it remains to bound  $\lambda_o$. This can be achieved
by an application  of the energy estimate from Lemma \ref{lem:energy} with $\theta=1$
and $a=0$ and Young's inequality. Indeed, we have
\begin{align*}
	\lambda_o
	\le
	c
	\Bigg[
	1+\biint_{Q_{8R}} \bigg[\frac{\abs{u}^{1+m}}{R^{\frac{1+m}{m}}} + |F|^{2}\bigg]
	\dx\dt
	\Bigg]^{\frac{d}{2m}} ,
\end{align*}
where $c=c(n,m,\nu,L)$.
Plugging this into the preceding estimate, we arrive at 
\begin{align*}
	\biint_{Q_{R}}
	|D\power{u}{m}|^{2+\epsilon} \d x\d t 
	&\le
	c
	\Bigg[
	1+\biint_{Q_{8R}} \bigg[\frac{|u|^{1+m}}{R^{\frac{1+m}{m}}} + |F|^{2}\bigg]
	\dx\dt
	\Bigg]^{\frac{\epsilon d}{2}}
	\biint_{Q_{2R}}
	|D\power{u}{m}|^{2} \dx\dt\\
        &\qquad\qquad+
	c\,\biint_{Q_{2R}}  |F|^{2+\epsilon} \dx\dt ,
\end{align*}
for a constant  $c=c(n,m,\nu,L)$. The asserted reverse H\"older inequality
\eqref{eq:higher-int} now follows by a standard covering argument. More
precisely, we cover $Q_R$ by finitely many cylinders of radius $\frac
R8$, apply the preceding estimate on each of the smaller cylinders and
take the sum. 
This yields the same estimate as above, but with integrals over
$Q_{2R}$ instead of $Q_{8R}$. This completes the proof of
Theorem~\ref{thm:higherint}.
\qed

\subsection{Proof of Corollary \ref{cor:higher-int}}
We consider a standard parabolic
cylinder $C_{2R}(z_o):=
B_{2R}(x_o)\times(t_o-(2R)^2,t_o+(2R)^2)\subseteq\Omega_T$ and rescale
the problem via
\begin{equation*}
	\left\{
	\begin{array}{c}
  	v(x,t):=u(x_o+Rx,t_o+R^2t) \\[5pt]
	\mathbf B(x,t,u,\xi):=R\,\mathbf A\big(x_o+Rx,t_o+R^2t,u,\tfrac1{R}\xi\big) \\[5pt]
	G(x,t):=R\,F(x_o+Rx,t_o+R^2t),
	\end{array}
	\right.
\end{equation*}
whenever $(x,t)\in C_{2}$ and $(u,\xi)\in\R^N\times\R^{Nn}$.
The rescaled function $v$ is a weak solution of the differential equation
\begin{equation*}
    \partial_tv-\mathrm{div}\,\mathbf B(x,t,v,D\power{v}{m})=\mathrm{div}\,G
    \qquad\mbox{in $\widetilde Q:=Q_{2^\frac{2m}{1+m}}\subseteq C_{2}$}
  \end{equation*}
  in the sense of Definition~\ref{def:weak_solution}. Moreover, the rescaled
  vector-field $\mathbf B$ satisfies
  assumptions~\eqref{growth}. Consequently, we can apply estimate
  \eqref{eq:higher-int} to $v$ on the cylinder $\widetilde Q$, which gives 
   \begin{align*}
	\biint_{\frac12\widetilde Q} &
          |D\power{v}{m}|^{2+\epsilon} \d x\d t \\
	&\le
	c\bigg[
	1+\biint_{\widetilde Q} \big[\abs{v}^{1+m} + |G|^2\big]\d x\d t 
	\bigg]^{\frac{\epsilon d}{2}}
	\biint_{\widetilde Q} |D\power{v}{m}|^{2} \dx\dt +
	c\,\biint_{\widetilde Q} |G|^{2+\epsilon} \dx\dt,    
  \end{align*}
  for every $\eps\in(0,\eps_1]$, where $c=c(n,m,\nu,L)$, and we
  abbreviated $\frac12\widetilde Q:=Q_{2^{\frac{m-1}{1+m}}}$.   
  Scaling back and using the fact that the cylinder $C_{\gamma R}$
  with $\gamma:=2^\frac{m-1}{2m}$ is contained in the re-scaled
  version of the cylinder $\frac12\widetilde Q$ 
  we deduce
   \begin{align*}
	R^{2+\eps} &
	\biint_{C_{\gamma R}(z_o)}
    |D\power{u}{m}|^{2+\epsilon} \d x\d t \\
	&\le
	c\,R^2\bigg[
	1+\biint_{C_{2R}(z_o)} \big[\abs{u}^{1+m} + R^2|F|^2\big] \d x\d t 
	\bigg]^{\frac{\epsilon d}{2}}
	\biint_{C_{2R}(z_o)}
	|D\power{u}{m}|^{2} \dx\dt \\
	&\quad +
	c\, R^{2+\epsilon}
	\biint_{C_{2R}(z_o)} |F|^{2+\epsilon} \dx\dt .
   \end{align*}
The asserted estimate on the pair $C_R$, $C_{2R}$ of standard parabolic cylinders now follows with a standard covering argument.  This 
finishes  the proof of Corollary~\ref{cor:higher-int}. \qed

\end{document}